\documentclass[12pt, leqno]{amsart}

\usepackage{graphicx}
\usepackage{amssymb,amsmath,amsthm,amscd}
\usepackage{mathrsfs}
\usepackage{enumerate}
\usepackage[usenames,dvipsnames]{color}
\usepackage[colorlinks=true, pdfstartview=FitV,
 linkcolor=blue,citecolor=blue,urlcolor=blue]{hyperref}
\usepackage[all]{xy}

\usepackage{verbatim}
\usepackage{oldgerm}
\usepackage[normalem]{ulem}  
\usepackage{marginnote}
\usepackage{xspace}
\numberwithin{equation}{section}
\allowdisplaybreaks[4]

\newcommand{\nc}{\newcommand}
\nc{\op}{\operatorname}

\theoremstyle{plain}
\newtheorem{lemma}{Lemma}[subsection]
\newtheorem{prop}[lemma]{Proposition}
\newtheorem{theorem}[lemma]{Theorem}
\newcommand{\Prop}{\begin{prop}}
\newcommand{\enprop}{\end{prop}}
\newcommand{\Lemma}{\begin{lemma}}
\newcommand{\enlemma}{\end{lemma}}
\newcommand{\Th}{\begin{theorem}}
\newcommand{\enth}{\end{theorem}}
\newtheorem{corollary}[lemma]{Corollary}
\newcommand{\Cor}{\begin{corollary}}
\newcommand{\encor}{\end{corollary}}
\newtheorem{definition}[lemma]{Definition}

\newcommand{\Def}{\begin{definition}}
\newcommand{\edf}{\end{definition}}
\newtheorem{sublemma}[lemma]{Sublemma}
\newcommand{\Sublemma}{\begin{sublemma}}
\newcommand{\ensub}{\end{sublemma}}

\theoremstyle{definition}
\newtheorem{remark}[lemma]{Remark}

\newtheorem{Convention}[lemma]{Convention}
\newcommand{\Conv}{\begin{Convention}}
\newcommand{\enconv}{\end{Convention}}
\newcommand{\Rem}{\begin{remark}}
\newcommand{\enrem}{\end{remark}}

\newcommand{\C}{\mathbb {C}}
\newcommand{\Q}{\mathbb {Q}}

\newcommand{\Z}{{\mathbb Z}}
\newcommand{\B}{{\mathbf{B}}}

\newcommand{\R}{{\rm R}}

\newcommand{\CC}{{\mathscr{C}}}
\newcommand{\one}{{\bf{1}}}
\newcommand{\seteq}{\mathbin{:=}}

\newcommand{\hd}{{\operatorname{hd}}}

\newcommand{\g}{{\mathfrak{g}}}

\newcommand{\Hom}{\operatorname{Hom}}

\newcommand{\isoto}[1][]{\mathop{\xrightarrow%
[{\raisebox{.3ex}[0ex][.3ex]{$\scriptstyle{#1}$}}]%
{{\raisebox{-.6ex}[0ex][-.6ex]{$\mspace{2mu}\sim\mspace{2mu}$}}}}}
\newcommand{\tensor}{\otimes}

\newenvironment{myequation}
{\relax\setlength{\arraycolsep}{1pt}\begin{eqnarray}}
{\end{eqnarray}}
\newenvironment{myequationn}
{\relax\setlength{\arraycolsep}{1pt}\begin{eqnarray*}}
{\end{eqnarray*}}

\nc{\eq}{\begin{myequation}}
\nc{\eneq}{\end{myequation}}
\nc{\eqn}{\begin{myequationn}}
\nc{\eneqn}{\end{myequationn}}
\newcommand{\on}{\operatorname}

\newcommand{\bni}{\be[{\rm(i)}]}
\newcommand{\bna}{\be[{\rm(a)}]}

\newcommand{\QED}{\end{proof}}
\newcommand{\Proof}{\begin{proof}}

\newcommand{\soplus}{\mathop{\mbox{\normalsize$\bigoplus$}}\limits}

\newcommand{\To}[1][{\hs{2ex}}]{\xrightarrow{\,#1\,}}
\newcommand{\id}{\on{id}}
\newcommand{\ba}{\begin{array}}
\newcommand{\ea}{\end{array}}

\newcommand{\bi}{\begin{enumerate}[{\rm(i)}]}

\newcommand{\monoto}{\rightarrowtail}

\newcommand{\set}[2]{\left\{#1 \mathbin{;} #2 \right\}}

\newcommand{\Mod}{\operatorname{Mod}}
\newcommand{\Modg}{\operatorname{{Mod}_{\mathrm{gr}}}}
\newcommand{\hs}{\hspace*}

\newcommand{\eqsub}{\begin{subequations}\begin{eqnarray}}
\newcommand{\eneqsub}{\end{eqnarray}\end{subequations}}

\newcommand{\ol}{\overline}

\nc{\univ}{\mathrm{univ}}
\nc{\la}{\lambda}
\nc{\lam}{\lambda}
\nc{\U}[1][\g]{U_q(#1)}
\nc{\te}{\tilde{e}}
\nc{\tei}{\tilde{e}_i}
\nc{\tf}{\tilde{f}}
\nc{\tfi}{\tilde{f}_i}
\nc{\tU}{\widetilde U_q(\g)}
\nc{\tE}{\tilde{E}}
\nc{\tF}{\widetilde{\F}}

\nc{\tk}{\tilde{k}}
\nc{\tkone}{\tk_{\ol{1}}}
\nc{\teone}{\tilde{e}_{\ol{1}}}
\nc{\tfone}{\tilde{f}_{\ol{1}}}

\nc{\teibar}{\tilde{e}_{\ol{i}}} \nc{\tfibar}{\tilde{f}_{\ol{i}}}
\nc{\tki}{{\tk}_{\ol {i}}}

\nc{\BZ}{{\mathbb{Z}}}
\nc{\al}{\alpha}
\nc{\qs}{{q_s}}
\nc{\qt}{{q_t}} 
\nc{\pstar}{{p^*}} 
\nc{\lan}{\langle}
\nc{\ran}{\rangle}
\nc{\re}{{\mathrm{re}}}
\nc{\wt}{\operatorname{wt}}
\nc{\ch}{\operatorname{ch}}
\nc{\Uf}[1][\g]{U^-_q(#1)}
\nc{\Ue}{U^+_q(\g)}
\nc{\eps}{\varepsilon}
\nc{\vphi}{\varphi}
\nc{\sphi}{\varphi^*}
\nc{\seps}{\varepsilon^*}

\nc{\nn}{\nonumber}

\nc{\vp}{\varpi}
\nc{\cls}{{\operatorname{cl}}}
\nc{\Wt}{{\operatorname{Wt}}}
\nc{\Us}{U'_q(\g)}
\nc{\La}{\Lambda}
\nc{\ro}{{\rm(}}
\nc{\rf}{{\rm)}}
\nc{\norm}{{\mathrm{norm}}}
\nc{\qbox}{\quad\mbox}
\nc{\braid}{{\mathfrak{B}}}
\nc{\Ad}{\operatorname{Ad}}
\nc{\Aut}{\operatorname{Aut}}
\nc{\dt}[1]{\tilde{\tilde #1}}
\nc{\Sn}{S^{{\mathrm{norm}}}}
\nc{\aff}{{\rm{aff}}}
\nc{\rk}{{\mathrm{rk}}}
\nc{\tP}{\widetilde{P}}
\nc{\tW}{\widetilde{W}}
\nc{\Dyn}{\mathrm{Dyn}}
\nc{\tD}{\widetilde{\Delta}}
\nc{\height}{{\operatorname{ht}}}
\nc{\bl}{\bigl(}
\nc{\br}{\bigr)}
\nc{\Hecke}{\mathrm{H}}
\nc{\HA}{\Hecke^{\mathrm{A}}}
\nc{\HB}{\Hecke^{\mathrm{B}}}
\newcommand{\scbul}{{\,\raise1pt\hbox{$\scriptscriptstyle\bullet$}\,}}
\nc{\vac}{{\phi}}
\nc{\Bt}{\B_\theta(\g)}
\nc{\be}{\begin{enumerate}}
\nc{\ee}{\end{enumerate}}
\nc{\low}{{\mathrm{low}}}
\nc{\upper}{{\mathrm{up}}}
\nc{\Zodd}{\Z_{\mathrm{odd}}}
\nc{\KA}{\on{K}^{\mathrm{A}}}
\nc{\KB}{\on{K}^{\mathrm{B}}}
\nc{\Res}{\on{Res}}
\nc{\Fc}[1][{n,m}]{\mathbf{F}_{#1}}
\nc{\tphi}{\tilde{\varphi}}
\nc{\CO}{\mathscr{O}}
\nc{\inte}{\mathrm{int}}
\nc{\Oint}{\mathcal{O}^{\ge0}_{\inte}}
\nc{\vs}{\vspace}
\nc{\tL}{\widetilde{L}}
\nc{\tu}{\tilde{u}}
\nc{\noi}{\noindent}
\nc{\heigh}{\mathfrak{t}}
\nc{\lowest}{\mathfrak{l}}
\nc{\rootl}{\mathsf{Q}}
\nc{\cl}{{\rm{cl}}}
\nc{\af}{{\rm{af}}}
\nc{\uqpg}{U'_q(\mathfrak g)}
\nc{\Oh}{\widehat{\mathcal{O}}}

\nc{\hV}{\widehat{V}}

\nc{\rd}{{}^*\mspace{-1mu}}

\newenvironment{rouge}
{\color{red}}
{}

\newcommand{\berm}{\begin{rouge}{}\marginnote{\fbox{\scshape\lowercase{M}}}{}}
\newcommand{\bermh}{\begin{rouge}{}\marginnote{\fbox{\scshape\lowercase{MH}}}{}}

\nc{\cmtm}[1]{\fbox{\ber{\rm M:}{\em{#1}}\er}}
\nc{\cmtmh}[1]{\fbox{\ber{\rm MH:}{\em{#1}}\er}}
\newenvironment{bleu}
{\color{blue}}
{}

\nc{\beb}{\begin{bleu}}
\nc{\eb}{\end{bleu}}

\nc{\KLR}{quiver Hecke algebra}
\nc{\KLRs}{quiver Hecke algebras}
\nc{\cor}{\mathbf{k}}
\nc{\cora}{{\cor(A)}}
\nc{\haut}{\mathrm{ht}}
\nc{\tens}{\mathop\otimes}
\nc{\gmod}{\mbox{-$\mathrm{gmod}$}}
\nc{\proj}{\mbox{-$\mathrm{proj}$}}
\nc{\gproj}{\mbox{-$\mathrm{gproj}$}}
\nc{\smod}{\mbox{-$\mathrm{mod}$}}
\nc{\nmod}{\mbox{-$\mathrm{nilmod}$}}

\nc{\h}{\mathfrak h}
\nc{\Rnorm}{R^{\rm{norm}}}
\nc{\Runiv}{R^{\rm{univ}}}
\nc{\Rren}{R^{\rm{ren}}} 
\nc{\Vhat}{\widehat{V}}
\nc{\F}{\mathcal{F}}
\def\AA{{\mathcal A}}

\def\T{{\mathcal T}}

\nc{\fd}[1][A]{\on{\mathrm{flat.dim}_{#1}}}
\nc{\bP}{{\mathbb{P}}}
\nc{\bPh}{\widehat{\mathbb{P}}}
\nc{\bK}{\widehat{\mathbb{K}}}
\nc{\bV}[1][{n}]{\widehat{V}^{\otimes{#1}}}
\nc{\bVK}[1][{n}]{\widehat{V}^{\otimes{#1}}_K}
\nc{\opp}{\mathrm{opp}}
\nc{\col}{\colon}
\nc{\bnum}{\be[{\rm(i)}]}
\nc{\oep}{\epsilon}
\nc{\qtext}{\quad\text}
\nc{\qtextq}[1]{\quad\text{#1}\quad}
\nc{\longtwoheadrightarrow}[1][]{\xymatrix{\ar@{->>}[r]^-{{#1}}&}}
\nc{\epiTo}[1][]{\longtwoheadrightarrow[{#1}]}
\nc{\epito}{\twoheadrightarrow}
\nc{\monoTo}[1][]{\xymatrix{\ar@{>->}[r]^-{{#1}}&}}
\nc{\sym}{\mathfrak{S}}
\nc{\inp}[1]{{({#1})_{\mathrm{n}}}}
\nc{\rtl}{\rootl}
\nc{\wtd}{\widetilde}
\nc{\etens}{\boxtimes}
\nc{\ds}[1]{\mathrm{d}(#1)}
\nc{\rmat}[1]{{\mathbf r}_{\mspace{-2mu}\raisebox{-.5ex}{${\scriptstyle{#1}}$}}}
\nc{\shc}{\mathcal{C}}
\nc{\Fct}{{\on{Fct}}}
\nc{\tC}{\widetilde{\shc}}
\nc{\Zp}{\Z_{\ge0}}
\nc{\tPhi}{\widetilde{\Phi}}
\nc{\tT}{{\tilde{\T}}}
\nc{\Ob}{\on{Ob}}
\nc{\bwr}{\mbox{\large$\wr$}}
\nc{\Img}{\on{Im}}
\nc{\Ab}{\mathcal{A}^{\mathrm{big}}}
\nc{\Sb}{\mathcal{S}^{\mathrm{big}}}
\nc{\As}{\mathcal{A}}
\nc{\Ss}{\mathcal{S}}
\nc{\ntens}{\widetilde{\otimes}}
\nc{\hR}{\widehat{R}}
\nc{\nconv}{\star}
\nc{\ts}{\tilde{s}}
\nc{\sho}{\mathcal{O}}
\nc{\bc}{\begin{cases}}
\nc{\ec}{\end{cases}}

\nc{\UA}{U_q'(A^{(1)}_{N-1})}
\nc{\UAtwo}{U_q'(A^{(2)}_{N-1})}
\nc{\KR}{R_K}
\nc{\cQ}{\mathcal{Q}}
\nc{\Irr}{\mathcal{I}rr}
\nc{\tQ}{\widetilde{\cQ}}
\nc{\bs}{\mathbf{s}}
\nc{\bL}{\mathbb{L}}
\nc{\KP}{{\mathrm{KP}}}
\nc{\db}{\mathsf{b}^*}
\nc{\bfa}{{\mathbf{a}}}
\nc{\bfc}{{\mathbf{c}}}
\nc{\Po}{(P_\cl)_0}
\renewcommand{\Im}{\op{Im}}
\nc{\mono}{\rightarrowtail}
\nc{\tr}{\on{tr}}
\nc{\K}{\on{K}}
\nc{\FQ}[1]{\F_Q^{(#1)}}
\nc{\bQ}{\ol{Q}}
\nc{\conv}{\mathbin{\mbox{\large $\circ$}}}
\nc{\uqm}{\uqpg\smod}

\nc{\Ft}{\mathcal{F}^{\rm{T}}}
\nc{\Pt}{P^{\rm{T}}}

\newcommand{\hconv}{\mathbin{\scalebox{.9}{$\nabla$}}}
\newcommand{\sconv}{\mathbin{\scalebox{.9}{$\Delta$}}}
\nc{\de}{\on{\textfrak{d}}}
\nc{\Rm}{R^{\mathrm{ren}}}
\nc{\ang}[1]{\lan{#1}\ran}
\nc{\ms}{\mspace}
\nc{\alj}{{\al^J}\ms{-3mu}}
\nc{\dcf}{duality coefficient\xspace}
\nc{\dcfs}{duality coefficients\xspace}
\nc{\snoi}{\smallskip\noindent}
\nc{\vpi}{\varpi}

\nc{\duality}{KLR-type quantum affine Schur-Weyl duality\ }

\newlength{\mylength}
\title[Quantum affine algebras of type A and B]{Monoidal categories of modules over quantum affine algebras of type A and B}

\author[Masaki Kashiwara]{Masaki Kashiwara}
\thanks{The research of M.\ Kashiwara
was supported by Grant-in-Aid for Scientific Research (B)
15H03608, Japan Society for the Promotion of Science.}
\address[Masaki Kashiwara]{Research Institute for Mathematical Sciences, Kyoto University,
Kyoto 606-8502, Japan \& Korea Institute for Advanced Study, Seoul 02455, Korea }
\email[Masaki Kashiwara]{masaki@kurims.kyoto-u.ac.jp}

\author[Myungho Kim]{Myungho Kim}
\address[Myungho Kim]{Department of Mathematics, Kyung Hee University, Seoul 02447, Korea}
\email[Myungho Kim]{mkim@khu.ac.kr}
\thanks{The research of M.\ Kim was supported by the National Research Foundation of
Korea(NRF) Grant funded by the Korea government(MSIP) (NRF-2017R1C1B2007824).}

\author[Se-jin Oh]{Se-jin Oh}
\thanks{ The research of S.-j.\ Oh was supported by the National Research Foundation of
Korea(NRF) Grant funded by the Korea government(MSIP) (NRF-2016R1C1B1010721).}
\address[Se-jin Oh]{Department of Mathematics, Ewha Womans University, Seoul 120-750, Korea}
\email[Se-jin Oh]{sejin092@gmail.com}



\begin{document}

\begin{abstract}
We construct an exact tensor functor from the category $\As$ of finite-dimensional graded modules over the quiver Hecke algebra of type $A_\infty$ 
to the category
$\mathscr C_{B^{(1)}_n}$ 
  of finite-dimensional integrable modules over the quantum affine algebra of type $B^{(1)}_n$. 
It factors through the category $\mathcal T_{2n}$, 
which is a localization of  $\As$.
As a result, this functor induces a ring isomorphism from the Grothendieck ring of $\mathcal T_{2n}$ (ignoring the gradings)
 to the Grothendieck ring of a subcategory $\mathscr C^{0}_{B^{(1)}_n}$ of $\mathscr C_{B^{(1)}_n}$.
 Moreover, it induces a bijection between the classes of simple objects. 
 Because the category $\mathcal T_{2n}$ is related to  categories  $\mathscr C^{0}_{A^{(t)}_{2n-1}}$ $(t=1,2)$  
 of the quantum affine algebras of type $A^{(t)}_{2n-1}$,
  we obtain an interesting connection between those categories of modules over quantum affine algebras of type $A$ and type $B$. 
  Namely, for each $t =1,2$,  there exists an isomorphism  between the Grothendieck ring of $\mathscr C^{0}_{A^{(t)}_{2n-1}}$  and the Grothendieck ring of $\mathscr C^{0}_{B^{(1)}_n}$, which induces a bijection between the classes of simple modules.
\end{abstract}

\maketitle
\tableofcontents

\section*{Introduction}
Let $\g$ be a Kac-Moody algebra of affine type and  let $\uqpg$ be the corresponding quantum affine algebra. 
Since the category  $\CC_\g$ of finite-dimensional integrable representations of $\uqpg$  has a rich structure, 
 it has been extensively investigated with various approaches (see for example, \cite{CP95,FR99,Her101,Kas02,Nak041}). 
In particular, when $\uqpg$ is the quantum affine algebra of type $A_{N-1}^{(1)}$, 
 there is a  functor, so called the  \emph{quantum affine Schur-Weyl duality functor}, 
  from the category of finite-dimensional modules over the affine Hecke algebra  to the category $\CC_{A_{N-1}^{(1)}}$ (\cite{CP96B, Che, GRV94}).
The \emph{KLR-type quantum affine Schur-Weyl duality}, which is introduced and developed in \cite{KKK13A, KKK13B, KKKO15, KKKO16, KO17},  is a wide generalization of the quantum affine Schur-Weyl duality.  
It provides a general procedure to obtain a functor 
  from the category of modules over a symmetric quiver Hecke algebra to the category $\CC_\g$ for the quantum affine algebra $\uqpg$ of \emph{arbitrary type}. 
Recall that a symmetric quiver Hecke algebra, also called a symmetric Khovanov-Lauda-Rouquier algebra, is a family of graded  algebras associated with
a quiver without loops. 
It is introduced as a generalization of the affine Hecke algebra in the context of categorification of quantum groups ({\cite{KL09, R08}}). 
In the present paper, we consider the quantum  affine algebra of type $B_n^{(1)}$ and 
we apply the general procedure of the  \duality  in order to obtain a functor from 
the category of finite-dimensional graded modules over the symmetric quiver Hecke algebra of type $A_\infty$ to the category $\CC_{B_n^{(1)}}$. 
As a consequence, we have a close relation between the categories of finite-dimensional integrable modules over quantum affine algebras of type $A$ and that of type $B$.

\medskip
For a detailed explanation of  our results, let us briefly recall the construction of \duality functors.
Suppose that an index set $J$  and a family $\{(V_j)_{X(j)}\}_{j \in J}$ of evaluation modules of good modules over $\uqpg$ are given. 
Then one can form a quiver $\Gamma^J$ such that the set of vertices is $J$ and the set of arrows is given by the order
of poles of the normalized $R$-matrices between the modules in the family. 
Then  one obtains a symmetric quiver Hecke algebra $R^J$ 
associated with the quiver $\Gamma^J$, 
and  constructs 
a tensor functor from the category of finite-dimensional graded modules over $R^J$ to the category $\CC_\g$. 
We call it the \emph{\duality functor}.

The case when $\g$ is of type $A_{N-1}^{(1)}$ and the family is given as $\{V(\vp_1)_{q^{2k}}\}_{k\in \Z}$ is thoroughly analyzed  in \cite{KKK13A}.
In this cases, the quiver $\Gamma^J$ is of type $A_\infty$ with the set of vertices $J=\Z$ and hence the corresponding  symmetric quiver Hecke algebra $R^J$ is of type $A_\infty$. 
Let $\As$ be the category of finite-dimensional graded modules over the symmetric quiver Hecke algebra of type $A_\infty$.
One of the main features of this case is that there exists a localization $\mathcal T_N$ of the category $\As$ such that  the duality functor factors through the category $\mathcal T_N$. 
Moreover, the category $\mathcal T_N$, 
as well as $\CC_{A_{N-1}^{(1)}}$, is a rigid tensor category. 
If we consider the smallest Serre subcategory $\CC_{A_{N-1}^{(1)}}^0$ of $\CC_{A_{N-1}^{(1)}}$  which is stable under taking tensor products and 
contains a sufficiently large family of fundamental representations in $\CC_{A_{N-1}^{(1)}}$ (for the precise definition see, Subsection \ref{subsec:HL category}), then the duality functor induces a ring isomorphism between the Grothendieck ring  $K(\mathcal T_N)_{q=1}$ and the Grothendieck ring $K(\CC_{A_{N-1}^{(1)}}^0)$ of the category $\CC_{A_{N-1}^{(1)}}^0$. 
Here the ring $K(\mathcal T_N)_{q=1}$ is obtained from the Grothendieck ring $K(\mathcal T_N) $ of the category $\mathcal T_N$ by ignoring the gradings.
Furthermore, the functor induces a bijection between 
 the sets of the classes of simple objects.

 In \cite{KKKO15},  it is shown  that the category $\mathcal T_N$ is related not only to the category $\CC_{A_{N-1}^{(1)}}^0$ but also to the category $\CC_{A_{N-1}^{(2)}}^0$ of the twisted quantum affine algebra $U_q\rq{} (A_{N-1}^{(2)})$ in a similar fashion. 
 More precisely, there is a dimension-preserving bijection between the fundamental representations in the category $\CC_{A_{N-1}^{(1)}}^{0}$ and those in the category $\CC_{A_{N-1}^{(2)}}^{0}$ 
which is induced from the \duality functors from the category $\As$ to the category $\CC_{A_{N-1}^{(t)}}^{0}$ $(t=1,2)$. Here the \duality functor
arises from the choice of the family $\{V(\vp_1)_{q^{2k}}\}_{k\in \Z}$ in $\CC_{A_{N-1}^{(t)}}^{0}$.
Moreover, this duality functor also factors through the category $\mathcal T_N$ and it induces a ring isomorphism between the Grothendieck rings  $K(\mathcal T_N)_{q=1}$  and $K(\CC_{A_{N-1}^{(2)}}^0)$, which is a bijection on the sets of classes of simple objects. 
Hence, in conclusion, the three Grothendieck rings $K(\CC_{A_{N-1}^{(1)}}^{0})$, $K(\CC_{A_{N-1}^{(2)}}^{0})$ and $K(\mathcal T_N)_{q=1}$ are isomorphic and those isomorphisms induce bijections between the classes of the simple objects.

\medskip
The main result of the present paper is that there is an exact tensor functor from the category $\mathcal T_{2n}$ to the category $\CC_{B_{n}^{(1)}}^0$ and  it induces a ring isomorphism between the Grothendieck rings, which is a bijection between the sets of  classes of simple objects. 
An interesting point  is that  the  same category $\mathcal T_{2n}$ plays a role, even we move from the quantum affine algebra of type $A$ to the quantum affine algebra of type $B$.
To obtain such a functor, first we present a family $\{ (V_j)_{X(j)}\}_{j \in \Z}$ of $U_q\rq{}(B_n^{(1)})$-modules, where $V_j$ is either $V(\varpi_1)$  or $V(\varpi_n)$ and $X:\Z \to \Q(q)^\times$ is a function such that the corresponding quiver $\Gamma$ is of type $A_\infty$ (for  the precise definition, see Subsection \ref{subsec:quiver}).
Since  each module $(V_j)_{X(j)}$ belongs to the category $\CC_{B_{n}^{(1)}}^0$,  the corresponding  \duality functor  $\F$ is a functor from the category $\As$  to the category $\CC_{B_{n}^{(1)}}^0$. 
Even though the choice of the family $\{ (V_j)_{X(j)}\}_{j \in \Z}$ is more complicated  in comparison to the case of the types $A_{N-1}^{(t)}$ $(t=1,2)$ , 
the duality functor $\F$ enjoys several similar properties: Recall that for each pair  $a \le b$ of integers, there is a simple module $L(a,b)$ over the symmetric quiver Hecke algebra of type $A_\infty$.  Then the functor  $\F$ sends the module $L(a,b)$  to zero if and only if $b-a+1 > 2n$, and  it sends $L(a,b)$ with $b-a+1=2n$ to the trivial representation of $U_q \rq{}(B_n^{(1)})$.
It follows that the kernel of the functor $\F$ coincides with the ones of the duality functors for type $A^{(t)}_{2n-1}$ $(t=1,2)$.
Moreover, we show that the functor $\F$ factors through the category $\mathcal T_{2n}$ again. 
To see this, it is enough to show that there exists a family $\{c_{i,j}(u,v) \}_{i,j \in J}$ of power series in two variables  satisfying certain conditions. 
We provide a general argument to construct such a family (Subsection \ref{subsec:factoring}), which is applicable to other cases including the cases of type $A_{N-1}^{(t)}$ $(t=1,2)$. 
We  further show that the resulting functor $\F\rq{} : \mathcal T_{2n} \to \CC_{B_{n}^{(1)}}^0$  induces a ring isomorphism from $K(\mathcal T_{2n})_{q=1}$ to $K(\CC_{B_{n}^{(1)}}^0)$, which is bijective between the sets of classes of simple objects.

\vs{1.5ex}
An immediate but  surprising consequence of the main result is that for each $t=1,2$, there exists a ring isomorphism between $K(\CC_{A_{2n-1}^{(t)}}^0)$ and $K(\CC_{B_{n}^{(1)}}^0)$, which induces a bijection between the classes of simple modules.  
We believe that the correspondence here between the quantum affine algebra $U_q\rq{}(B_n^{(1)})$ and $U_q\rq{}(A_{2n-1}^{(2)})$ is related  to the fact that they are  Langlands dual to each other; i.e.,  their  Dynkin diagrams are obtained from each other by changing the directions of arrows. 
This speculation is supported heuristically by the studies on some small subcategories of $\CC_\g^{0}$. 
Recall that in \cite{KKK13B, KKKO16}, certain tensor subcategories $\CC_Q^{(t)}$ of $\CC_\g$, where  $\g=A_{2n-1}^{(t)}$ or $D_{n+1}^{(t)}$ $(t=1,2)$, are studied using the \duality functors and it turns out that for each $t=1,2$, there is a ring isomorphism between the Grothendieck rings $K(R^{\g_0} \gmod)_{q=1}$ and $K(\CC_Q^{(t)})$, where $R^{\g_0} \gmod$ is the category of finite-dimensional graded modules over the symmetric quiver Hecke algebra of type $\g_0= A_{2n-1}$ or $D_{n+1}$. 
It induces also a bijection between the sets of classes of simple modules. 
In \cite{KO17}, the Langlands dual version  $\CC_{\mathscr Q}$ of the category $\CC_{Q}$ is investigated. 
That is, there exist subcategories   $\CC_{\mathscr Q}$ of $\CC_\g$ where  $\g=B_n^{(1)}$ or $C_{n}^{(1)}$,  and there is a ring isomorphism between the Grothendieck rings $K(R^{\g_0} \gmod)_{q=1}$ and $K(\CC_{\mathscr Q})$, which induces a bijection between the sets of classes of simple modules.
As a consequence, we have ring isomorphisms between the Grothendieck rings $K(\CC_Q)$,  $K(\CC_{\mathscr Q})$, and $K(R^{\g_0} \gmod)_{q=1}$ which are bijective between the sets of classes of simple modules.
Thus the main result of the present paper can be regarded as a global version of \cite{KO17} in the case of type $A$ and type $B$ quantum affine algebras.

Recall that in  \cite{FH11A} and \cite{FR99}, some interesting connections between the categories of finite-dimensional integrable representations  over a quantum affine algebra and that of its  Langlands  dual are suggested and studied. Since the category $\CC_{\g}^{0}$ can be regarded as a skeleton of the category $\CC_{\g}$ of all the finite-dimensional integrable representations, 
the correspondence between the categories $\CC_{A_{2n-1}^{(2)}}^{0}$ and $\CC_{B_{n}^{(1)}}^{0}$ in the present paper  looks relevant to their works.
 It would be interesting to find  a connection between the result in this paper and those in  \cite{ FH11A, FR99}.

\medskip
This paper is organized as follows.
In the first section, we recall some necessary materials such as symmetric quiver Hecke algebras, quantum affine algebras, $R$-matrices and  \duality functors.
In the second section, we present and study the \duality functor $\F$ from  $\As$ to $\CC_{B_n^{(1)}}^{(0)}$. We show that $\F$ factors through the category $\mathcal T_{2n}$ and the resulting functor induces an isomorphism between the Grothendieck rings, which is a bijection between the sets of classes of simple objects.
In the last section, we record the images inside the category $\CC_{B_{n}^{(1)}}^{0}$ of the fundamental representations in $\CC_{A_{2n-1}^{(t)}}^{0}$ $(t=1,2)$ under the correspondence between the Grothendieck rings.

\section{Symmetric quiver Hecke algebras and quantum affine algebras} \label{sec:backgrounds}
\subsection{Cartan datum and quantum groups}\label{subsec:Cartan}
In this subsection, we recall the definition of quantum groups. Let $I$
be an index set. A \emph{Cartan datum} is a sextuple $\bl A,P,
\Pi,P^{\vee},\Pi^{\vee},(\cdot\,,\,\cdot)\br$ 
consisting of
\begin{enumerate}[(a)]
\item an integer-valued matrix $A=(a_{ij})_{i,j \in I}$,
called \emph{the symmetrizable generalized Cartan matrix},
 which satisfies
\bni
\item $a_{ii} = 2$ $(i \in I)$,
\item $a_{ij} \le 0 $ $(i \neq j)$,
\item $a_{ij}=0$ if $a_{ji}=0$ $(i,j \in I)$,
\end{enumerate}
\item a free abelian group $P$, called the \emph{weight lattice},
\item $\Pi= \{ \alpha_i \in P \mid \ i \in I \}$, called
the set of \emph{simple roots},
\item $P^{\vee}\seteq\Hom(P, \Z)$, called the \emph{co-weight lattice},
\item $\Pi^{\vee}= \{ h_i \ | \ i \in I \}\subset P^{\vee}$, called
the set of \emph{simple coroots},
\item a symmetric $\Q$-valued bilinear form $(\cdot\,,\,\cdot)$ on $P$,
\end{enumerate}

\noindent
satisfying the following properties:
\bnum
\item $\langle h_i,\alpha_j \rangle = a_{ij}$ for all $i,j \in I$,
\item $\Pi$ is linearly independent,
\item for each $i \in I$, there exists $\Lambda_i \in P$ such that
           $\langle h_j, \Lambda_i \rangle =\delta_{ij}$ for all $j \in I$,
\item  $(\al_i,\al_i)\in\Q_{>0}$ for any $i\in I$, 
\item for any $\lambda\in P $ and $i \in I$,
one has
$$\lan h_i,\lambda\ran=
\dfrac{2(\alpha_i,\lambda)}{(\alpha_i,\alpha_i)}.$$ 
\ee
We call $\Lambda_i$ the \emph{fundamental weights}.
The free abelian group $\rootl\seteq\soplus_{i \in I} \Z \alpha_i$ is called the
\emph{root lattice}. Set $\rootl^{+}= \sum_{i \in I} \Z_{\ge 0}
\alpha_i\subset\rootl$ and $\rootl^{-}= \sum_{i \in I} \Z_{\le0}
\alpha_i\subset\rootl$. For $\beta=\sum_{i\in I}m_i\al_i\in\rootl$,
we set
$|\beta|=\sum_{i\in I}|m_i|$.

Let $d$  be the smallest positive integer such that $d \dfrac{(\alpha_i,\alpha_i)}{2} \in \Z$ for all $i \in I$.
Let $q$ be an indeterminate. For each $i \in I$, set $q_i = q^{\frac{(\alpha_i,\alpha_i)}{2}}\in \Q(q^{1/d})$. 

\begin{definition} \label{def:qgroup}
The {\em quantum group} $U_q(\g)$  associated   with  a Cartan datum
$\bl A,  P,\Pi,P^{\vee}, \Pi^{\vee},(\cdot\,,\,\cdot)\br$ 
is the  algebra  over $\mathbb
Q(q^{1/d})$ generated by $e_i,f_i$ $(i \in I)$ and $q^{h}$ $(h \in
d^{-1}P^\vee)$ satisfying following relations:
\begin{equation*}
\begin{aligned}
& q^0=1,\quad q^{h} q^{h'}=q^{h+h'}\ \text{for $h,h' \in d^{-1} P^\vee$,}\\
& q^{h}e_i q^{-h}= q^{\lan h, \alpha_i\ran} e_i, \quad
        q^{h}f_i q^{-h} = q^{-\lan h, \alpha_i\ran} f_i \quad
        \text{for $ h \in d^{-1}P^\vee$, $i \in I$,} \\
& e_if_j - f_je_i = \delta_{ij} \dfrac{K_i -K^{-1}_i}{q_i- q^{-1}_i
}\quad\text{where $K_i=q^{\frac{(\alpha_i,\alpha_i)}{2} h_i}$, }\\
& \sum^{1-a_{ij}}_{r=0} (-1)^r
        e^{(1-a_{ij}-r)}_i e_j e^{(r)}_i
 =\sum^{1-a_{ij}}_{r=0} (-1)^r
        f^{(1-a_{ij}-r)}_if_j f^{(r)}_i=0 \quad \text{ if } i \ne j.
\end{aligned}
\end{equation*}
\end{definition}

Here, we set $[n]_i =\dfrac{ q^n_{i} - q^{-n}_{i} }{ q_{i} - q^{-1}_{i} },\quad
  [n]_i! = \prod^{n}_{k=1} [k]_i$, $e_i^{(m)}=\dfrac{e_i^n}{[n]_i!}$ and $f_i^{(n)}=\dfrac{f_i^n}{[n]_i!}$
  for all $n \in \Z_{\ge 0}$, $i \in I$.

\subsection{Quiver Hecke algebras}
We recall the definition of \KLRs\ associated with a given
Cartan datum $\bl A, P, \Pi, P^{\vee}, \Pi^{\vee},(\cdot\,,\,\cdot)\br$
such that $(\al_i,\al_i)\in2\Z$ for all $i\in I$. 

Let $\cor$ be a commutative ring.
 Let us take  a family of polynomials $(Q_{i,j})_{i,j\in I}$ in $\cor[u,v]$
which satisfies
\eq
&&\left\{\parbox{\mylength}{
\bna
\item $Q_{i,j}(u,v)=Q_{j,i}(v,u)$ for any $i,j\in I$,
\item $Q_{i,j}(u,v)=0$ for $i=j\in I$,
\item for $i\not=j$, we have
$$
Q_{i,j}(u,v) =
\sum\limits_{ \substack{ (p,q)\in \Z^2_{\ge0} 
\\ (\al_i , \al_i)p+(\al_j , \al_j)q=-2(\al_i , \al_j)}}
t_{i,j;p,q} u^p v^q
$$
with $t_{i,j;p,q}\in\cor$, $t_{i,j;p,q}=t_{j,i;q,p}$ and $t_{i,j:-a_{ij},0} \in \cor^{\times}$. \ee}\right.
\label{eq:Q}
\eneq

We denote by
$\sym_{n} = \langle s_1, \ldots, s_{n-1} \rangle$ the symmetric group
on $n$ letters, where $s_i\seteq (i, i+1)$ is the transposition of $i$ and $i+1$.
Then $\sym_n$ acts on $I^n$ by place permutations.

For $n \in \Z_{\ge 0}$ and $\beta \in \rootl^+$ such that $|\beta| = n$, we set
$$I^{\beta} = \set{\nu = (\nu_1, \ldots, \nu_n) \in I^{n}}%
{ \alpha_{\nu_1} + \cdots + \alpha_{\nu_n} = \beta }.$$

\begin{definition}
For $\beta \in \rootl^+$ with $|\beta|=n$, the {\em
quiver Hecke algebra}  $R(\beta)$  at $\beta$ associated
with a Cartan datum $\bl A,P, \Pi,P^{\vee},\Pi^{\vee},(\cdot\,,\,\cdot)\br$ 
 and a matrix
$(Q_{i,j})_{i,j \in I}$ is the $\cor$-algebra generated by
the elements $\{ e(\nu) \}_{\nu \in  I^{\beta}}$, $ \{x_k \}_{1 \le
k \le n}$, $\{ \tau_m \}_{1 \le m \le n-1}$ satisfying the following
defining relations:

\begin{align*} 
& e(\nu) e(\nu') = \delta_{\nu, \nu'} e(\nu), \ \
\sum_{\nu \in  I^{\beta} } e(\nu) = 1, \allowdisplaybreaks\\
& x_{k} x_{m} = x_{m} x_{k}, \ \ x_{k} e(\nu) = e(\nu) x_{k}, \allowdisplaybreaks\\
& \tau_{m} e(\nu) = e(s_{m}(\nu)) \tau_{m}, \ \ \tau_{k} \tau_{m} =
\tau_{m} \tau_{k} \ \ \text{if} \ |k-m|>1, \allowdisplaybreaks\\
& \tau_{k}^2 e(\nu) = Q_{\nu_{k}, \nu_{k+1}} (x_{k}, x_{k+1})
e(\nu), \allowdisplaybreaks\\
& (\tau_{k} x_{m} - x_{s_k(m)} \tau_{k}) e(\nu) = \begin{cases}
-e(\nu) \ \ & \text{if} \ m=k, \nu_{k} = \nu_{k+1}, \\
e(\nu) \ \ & \text{if} \ m=k+1, \nu_{k}=\nu_{k+1}, \\
0 \ \ & \text{otherwise},
\end{cases} \allowdisplaybreaks\\
& (\tau_{k+1} \tau_{k} \tau_{k+1}-\tau_{k} \tau_{k+1} \tau_{k}) e(\nu)\\
& =\begin{cases} \dfrac{Q_{\nu_{k}, \nu_{k+1}}(x_{k},
x_{k+1}) - Q_{\nu_{k}, \nu_{k+1}}(x_{k+2}, x_{k+1})} {x_{k} -
x_{k+2}}e(\nu) \ \ & \text{if} \
\nu_{k} = \nu_{k+2}, \\
0 \ \ & \text{otherwise}.
\end{cases}
\end{align*}
\end{definition}

Note that $R(\beta)$ is a  $\Z$-graded algebra provided with
\begin{equation*} \label{eq:Z-grading}
\deg e(\nu) =0, \quad \deg\, x_{k} e(\nu) = (\alpha_{\nu_k}
, \alpha_{\nu_k}), \quad\deg\, \tau_{l} e(\nu) = -
(\alpha_{\nu_l} , \alpha_{\nu_{l+1}}).
\end{equation*}
Let us denote  by 
$\Mod(R(\beta))$ the category of  $R(\beta)$-modules and by 
$\Modg (R(\beta))$ the category of graded  $R(\beta)$-modules.
The category of graded $R(\beta)$-modules which are finite-dimensional over $\cor$ is denoted by $R(\beta)\gmod$.

In this paper, an $R(\beta)$-module means a graded $R(\beta)$-module, unless stated otherwise.

 For a graded $R(\beta)$-module $M=\bigoplus_{k \in \Z} M_k$, we define
$qM =\bigoplus_{k \in \Z} (qM)_k$, where
 \begin{align*}
 (qM)_k = M_{k-1} & \ (k \in \Z).
 \end{align*}
We call $q$ the \emph{grading shift functor} on the category of
graded $R(\beta)$-modules.

\medskip
For $\beta, \gamma \in \rootl^+$ with $|\beta|=m$, $|\gamma|= n$,
 set
$$e(\beta,\gamma)=\displaystyle\sum_{\substack%
{\nu \in I^{m+n}, \\ (\nu_1, \ldots ,\nu_m) \in I^{\beta},\\
(\nu_{m+1}, \ldots ,\nu_{m+n}) \in I^{\gamma}}} e(\nu). $$ 
Then $e(\beta,\gamma)$ is an idempotent
element of $R(\beta+\gamma)$.
Let 
\eq
R( \beta)\tens R( \gamma  )\to e(\beta,\gamma)R(
\beta+\gamma)e(\beta,\gamma) \label{eq:embedding} \eneq be the
$\cor$-algebra homomorphism given by
\begin{equation*}
\begin{aligned}
& e(\mu)\tens e(\nu)\mapsto e(\mu*\nu) \ \ (\mu\in I^{\beta}),\\
& x_k\tens 1\mapsto x_ke(\beta,\gamma) \ \  (1\le k\le m), \\
& 1\tens x_k\mapsto x_{m+k}e(\beta,\gamma) \ \  (1\le k\le n), \\
& \tau_k\tens 1\mapsto \tau_ke(\beta,\gamma) \ \  (1\le k<m), \\
& 1\tens \tau_k\mapsto \tau_{m+k}e(\beta,\gamma) \ \  (1\le k<n), 
\end{aligned}
\end{equation*}
where  $\mu*\nu$ is the concatenation of $\mu$ and $\nu$; i.e.,
$\mu*\nu=(\mu_1,\ldots,\mu_m,\nu_1,\ldots,\nu_n)$.

\medskip
For an  $R(\beta)$-module $M$ and  an   $R(\gamma)$-module $N$,
we define the \emph{convolution product}
$M\conv N$ by
$$M\conv N \seteq R(\beta + \gamma) e(\beta,\gamma)
\tens_{R(\beta )\otimes R( \gamma)}(M\otimes N). $$

\subsection{R-matrices for quiver Hecke algebra}
 For $|\beta|=n$ and $1\le a<n$,  we define $\vphi_a\in R( \beta)$
by
\eq&&\ba{l} \label{eq:intertwiner}
  \vphi_a e(\nu)=
\begin{cases}
  \bl\tau_ax_a-x_{a}\tau_a\br e(\nu) & \text{if $\nu_a=\nu_{a+1}$,} \\[2ex]
\tau_ae(\nu)& \text{otherwise.}
\end{cases}
\label{def:int} \ea \eneq
They are called the {\em intertwiners}.
Since $\{\vphi_k\}_{1\le k\le n-1}$ satisfies the braid relation,
we have a well-defined element $\vphi_w \in R(\beta)$ for each $w \in \sym_n$.

 For $m,n\in\Z_{\ge0}$,
we set
\eqn
&&\sym_{m,n} \seteq \set{w\in\sym_{m+n}}{\text{ $w(i)<w(i+1)$ for any $i\not=m$}}.
\eneqn
For example,
\eqn
&&w[{m,n}](k)=\begin{cases}k+n&\text{if $1\le k\le m$,}\\
k-m&\text{if $m<k\le m+n$}\end{cases}
\eneqn
is an element in $\sym_{m,n}$.

Let $\beta,\gamma\in \rtl^+$ with $|\beta|=m$, $|\gamma|=n$ and let
$M$ be an 
 $R(\beta)$-module and let $N$ be an $R(\gamma)$-module. 
 Then the
map
$$M\tens N\to q^{(\beta,\gamma)-2\inp{\beta,\gamma}}N\conv M$$ given by
$$u\tens v\longmapsto \vphi_{w[n,m]}(v\tens u)$$
is an  $R( \beta,\gamma )$-module homomorphism by \cite[Lemma 1.3.1]{KKK13A}, and it extends
to an $R( \beta +\gamma)$-module  homomorphism \eq &&\R_{M,N}\col
M\conv N\To q^{(\beta,\gamma)-2\inp{\beta,\gamma}}N\conv M, \eneq
where
the symmetric bilinear form $\inp{{\scbul,\scbul}}$ on $\rtl$ is given by
$\inp{\al_i,\al_j}=\delta_{ij}.$

\Def A \KLR\ $R( \beta )$ is {\em symmetric} if
$Q_{i,j}(u,v)$ is a polynomial in $\cor[u-v]$ for all $i,j\in I$.
\edf

\medskip
{\em From now on, we assume that \KLRs\ are symmetric.}
Then, the generalized Cartan matrix $A=(a_{ij})_{i,j \in I}$ is symmetric.
We assume then $(\al_i,\al_j)=a_{i,j}$.
Let $z$ be an indeterminate  which is  homogeneous of degree $2$, and
let $\psi_z$ be the algebra homomorphism
\eqn
&&\psi_z\col R( \beta )\to \cor[z]\tens R( \beta )
\eneqn
given by
$$\psi_z(x_k)=x_k+z,\quad\psi_z(\tau_k)=\tau_k, \quad\psi_z(e(\nu))=e(\nu).$$

For an $R( \beta )$-module $M$, we denote by $M_z$ the
$\bl\cor[z]\tens R( \beta )\br$-module $\cor[z]\tens M$ with the
action of $R( \beta )$ twisted by $\psi_z$. Namely,
\begin{equation}\label{eq:spt1}
\begin{aligned}
& e(\nu)(a\tens u)=a\tens e(\nu)u, \\
& x_k(a\tens u)=(za)\tens u+a\tens (x_ku), \\
& \tau_k(a\tens u)=a\tens(\tau_k u)
\end{aligned}
\end{equation}
for  $\nu\in I^\beta$,  $a\in \cor[z]$ and
$u\in M$.
For $u\in M$, we sometimes denote by $u_z$ the corresponding element $1\tens u$ of
the $R( \beta )$-module  $M_z$.

For a non-zero $R(\beta)$-module $M$ and a non-zero $R(\gamma)$-module $N$,
\eq&&\parbox{71ex}{%
let $s$ be the order of  zero of $R_{M_z,N_{z'}}\col  M_z\conv
N_{z'}\To q^{(\beta,\gamma)-2\inp{\beta,\gamma}}N_{z'}\conv M_z$,
i.e., the largest non-negative integer such that the image of
$R_{M_z,N_{z'}}$ is contained in $(z'-z)^s
q^{(\beta,\gamma)-2\inp{\beta,\gamma}}N_{z'}\conv M_z$.}
\label{def:s} \eneq Note that \cite[Proposition 1.4.4 (iii)]{KKK13A} shows
that
 such an $s$  exists.

\Def For a non-zero $R(\beta)$-module $M$ and a non-zero $R(\gamma)$-module
$N$,
we set $$\Lambda(M,N)\seteq - (\beta,\gamma)+2\inp{\beta,\gamma}-2s, \qquad \Rren_{M_z,N_{z'}}:= (z'-z)^{-s}R_{M_z,N_{z'}}$$ 
and define
$$\rmat{M,N}\col M\conv N\to q^{-\Lambda(M,N)}N\conv M$$
by 
 $$\rmat{M,N} = \bl \Rren_{M_z,N_{z'}} \br\vert_{z=z'=0}.$$
\edf
 By \cite[Proposition 1.4.4 (ii)]{KKK13A}, the morphism $\rmat{M,N}$ does not vanish
if $M$ and $N$ are non-zero.

\Lemma [\cite{KKKO15}] \label{lem:even}
Let $M$ and $N$ be simple $R$-modules.
Then we have
\bnum
   \item $\Lambda(M,N)+\Lambda(N,M) \in 2 \Z_{\ge 0}.$
   \item If $\Lambda(M,N)+\Lambda(N,M) =2m$ for some $m \in \Z_{\ge 0}$, then
  \eqn
\Rm_{M_z,N} \circ \Rm_{N,M_z}=z^m \id_{N \circ M_z} \quad
\text{and} \quad
\Rm_{N,M_z} \circ \Rm_{M_z,N}=z^m \id_{M_z \circ  N}
  \eneqn
  up to constant multiples.
\end{enumerate}
\enlemma

For simple modules $M$ and $N$, we set
\eqn
\de(M,N) = \frac{1}{2}\bl\Lambda(M,N)+\Lambda(N,M)\br\in\Z_{\ge0}.
\eneqn

A simple module $M \in R\gmod$ is called \emph{real} if $M\conv M$ is simple.

\Lemma[\cite{KKKO14}] \label{lem:commute_equiv}
 Let $M$ and $N$ be simple modules in $R \gmod$, and assume that one of them is real. Then 
\bnum
  \item $M \conv N$ and $N \conv M$ have simple socles and simple heads.
  \item Moreover, $\Im(\rmat{M,N})$ is equal to the head of $M \conv N$
and the socle of $N \conv M$.
Similarly,
$\Im(\rmat{N,M})$ is equal to the head of $N \conv M$
and the socle of $M \conv N$  \ro up to grading shifts\rf.
\end{enumerate}
\enlemma
For $M, N\in  R\gmod$, 
we denote  by $M \hconv N$ and $M \sconv N$ the head and the socle
of $M \conv N$, respectively. 

The number $\de(M,N)$ measures the degree of complexity of $M\conv N$
as seen in the following lemma.

\Lemma[{\cite[Proposition 4.14]{KKKO15mm}}]\label{lem:linked}
Let $M$ and $N$ be simple modules in $R \gmod$, and assume that one of them is real.

\bnum
\item $M\conv N$ is simple if and only if
$\de(M,N)=0$.
\item If $\de(M,N)=1$, then $M\conv N$ has length $2$, and there exists an exact sequence
$$0\to M\sconv N\to M\conv N\to M\hconv N\to 0.$$
\ee
\enlemma

\subsection{Quantum affine algebras}\label{sec:quantum affine algebras}
In this subsection, we briefly review the representation theory of finite-dimensional integrable modules over 
quantum affine algebras following \cite{AK, Kas02}. When concerned
with quantum affine algebras, we take the algebraic closure of
$\C(q)$ in $\cup_{m >0}\C((q^{1/m}))$ as  the  base field
$\cor$.

Let $I$ be an index set and  let $A=(a_{ij})_{i,j\in I}$ 
be a generalized Cartan matrix of affine type. 
 We choose $0\in I$ as the leftmost vertices in the tables
in \cite[{pages 54, 55}]{Kac} except $A^{(2)}_{2n}$-case in which
case we take the longest simple root as $\al_0$. Set $I_0
=I\setminus\{0\}$.

The weight lattice $P$ is given by
\begin{align*}
  P = \Bigl(\soplus_{i\in I}\Z \La_i\Bigr) \oplus \Z \delta,
\end{align*}
and the simple roots are given by
$$\al_i=\sum_{j\in I}a_{ji}\La_j \ \text{for } \ i \in I_0, \quad \text{and} \quad
\al_0=\sum_{j\in I}a_{j0}\La_j+\delta.$$
The weight $\delta$ is called the \emph{imaginary root}.
There exist $d_i\in\Z_{>0}$ such that
$$\delta=\sum_{i\in I}d_i\al_i.$$
Note that $d_i=1$ for $i=0$.
The simple coroots $h_i\in P^\vee$ are given by
\eqn
&& \lan h_i,\La_j\ran=\delta_{ij},\quad \lan h_i,\delta\ran=0.
\eneqn

Let $c=\sum_{i\in I}c_i h_i$ be a unique element such that
$c_i\in\Z_{>0}$ and
$$\Z\, c=\set{h\in\soplus\nolimits_{i\in I}\Z h_i}{\text{$\lan h,\al_i\ran=0$ for any $i\in I$}}.$$
We normalize the $\Q$-valued symmetric bilinear form  $(\scbul,\scbul)$
on $P$
by 
 \eq \label{eq:normal}
(\delta, \la)=\lan c,\la\ran \quad \text{for any } \ \la\in P.\eneq

Let us denote by $U_q(\g)$ the quantum group associated with the affine Cartan datum $\bl A,P, \Pi,P^{\vee},\Pi^{\vee},(\cdot\,,\,\cdot)\br$. 
 We denote by $\uqpg$ the subalgebra of $U_q(\mathfrak{g})$ generated by $e_i,f_i,K_i^{\pm1}$ for $i \in I$.
 We call $\uqpg$ the \emph{quantum affine algebra} associated with 
 the generalized Cartan matrix $A$. 
Let us denote by $\Mod(\uqpg)$ the category of left modules of $\uqpg$.

The algebra $\uqpg$ has a Hopf algebra structure with the following coproduct $\Delta$, counit $\varepsilon$, and  antipode $S$:
\eqn
&&\Delta(K_i)=K_i\tens K_i, \quad \Delta(e_i)=e_i\tens K_i^{-1}+1\tens e_i,
\quad \Delta(f_i)=f_i\tens 1+K_i\tens f_i, \\
&&\varepsilon(K_i) = 1, \quad \varepsilon(e_i) = \varepsilon(f_i) = 0, \\
&&S(K_i) = {K_i}^{-1}, \quad S(e_i) = - e_i K_i, \quad S(f_i) = - K_i^{-1}f _i.
\eneqn

Set
$$P_\cl=P/\Z \delta$$
and call it the {\em classical weight lattice}. Let $\cl\col P\to
P_\cl$  denote the projection. 
 Then $P_\cl=\soplus\nolimits_{i\in I}\Z\,\cl(\La_i)$.
Set $P^0_\cl = \set{\la\in P_\cl}{\lan c,\la\ran=0}\subset P_\cl$.

Let us denote by $W$ the Weyl group, i.e., the subgroup of $\Aut(P)$
generated by the simple reflections $s_i$ ($i\in I$),
where $s_i(\la)=\la-\ang{h_i,\la}$ ($\la\in P)$. Then $W$ acts also
on  $P_\cl$ and $P^0_\cl$.

A $\uqpg$-module $M$ is called an {\em integrable module} if

\hs{5ex}\parbox[t]{70ex}{
 \bna
\item $M$ has a weight space decomposition
$$M = \bigoplus_{\lambda \in P_\cl} M_\lambda,$$
where  $M_{\lambda}= \set{ u \in M }{%
\text{$K_i u =q_i^{\lan h_i , \lambda \ran} u$ for all $i\in I$}}$,
\item  the actions of
 $e_i$ and $f_i$ on $M$ are locally nilpotent for any $i\in I$.
\ee}

Let us denote by $\CC_\g$
the abelian tensor category of finite-dimensional integrable $\uqpg$-modules.

If $M$ is a  simple module in $\CC_\g$,
then there exists a non-zero vector $u\in M$
of weight $\la\in P_\cl^0$
such that $\la $ is dominant (i.e.,
$\lan h_i,\la\ran\ge0$ for any $i\in I_0$)
and all the weights of $M$ lie in
$\la-\sum_{i\in I_0}\Z_{\ge0} \cl(\al_i)$.
Such a $\la$ is unique and called 
the {\em dominant extremal weight} of $M$.
Moreover, we have $\dim M_\la=1$, and a non-zero vector of $M_\la$
is called a {\em dominant extremal vector} of $M$.

Let $M$ be an integrable $\uqpg$-module. Then the {\em affinization}
$M_\aff$ of $M$ is the space
$M_\aff=\cor[z,z^{-1}] \tens M$
 with a $\cor[z,z^{-1}] \otimes \uqpg$-action given by 
$$e_i(u_z)=z^{\delta_{i,0}}(e_iu)_z, \quad
f_i(u_z)=z^{-\delta_{i,0}}(f_iu)_z, \quad
K_i(u_z)=(K_iu)_z,
$$
 where $u_z$  denotes  the element $ \one \tens u \in M_\aff$ for $u \in M$.
We 
denote the action 
of  $z$ on $M_\aff$  by $z_M$.  We sometimes write
$(M_z,z)$ for $(M_\aff,z_M)$.

For $x \in \cor^\times$, we define
 $$M_x   \seteq M_\aff/(z_M-x)M_\aff.$$
We sometimes call $x$ the {\em spectral parameter}.

Set $\varpi_i=\gcd(c_0,c_i)^{-1}\cl\bl c_0\Lambda_i-c_i\Lambda_0\br \in P_\cl$
for $i \in I_0$.
Then $\{\varpi_i\}_{i\in I_0}$ forms a basis of $P^0_\cl$.
We call $\varpi_i$ a {\em \ro level $0$\rf\ fundamental weight}.

Then for any $i\in I_0$, there exists a non-zero
$\uqpg$-module $M$ in $\CC_\g$ satisfying the following properties:
we can take $u_\la\in M_\la$ for each $\la\in W\varpi_i\subset P_\cl$
such that
\eqn
&&\parbox{\mylength}{
\bna 
\item
if $j\in I$ and $\la\in W\varpi_i$ satisfy
$\langle h_j,\lambda\rangle\ge 0$, then
$e_ju_\la=0$ and $f_j^{(\langle h_j,\lambda\rangle)}u_\la=u_{s_j\la}$, 
\item
if $j\in I$ and $\la\in W\varpi_i$ satisfy
$\langle h_j,\lambda\rangle\le 0$, then
$f_ju_\la=0$ and $e_j^{( -  \langle h_j,\lambda\rangle)} u_\la=u_{s_j\la}$,
\item
$M$ is generated by $u_{\varpi_i}$.
\ee
}
\eneqn
Then $M$ is a simple module with a dominant extremal weight $\vpi_i$
and  it is
unique up to an isomorphism (\cite{Kas02}).
We call $M$ the fundamental representation with 
dominant extremal weight $\varpi_i$, and denote it by $V(\vpi_i)$.

 With a slight abuse of terminology, we also call $V(\varpi_i)_x$  ($x\in\cor^\times$)
a fundamental representation. 

\bigskip
 If a $\uqpg$-module $M \in \CC_\g$ has a {\it  bar
involution}, a crystal basis with {\it simple crystal graph}
and a {\it lower global basis}, then we say that $M$ is a {\em
good module}.
For the precise definition, see \cite[\S\;8]{Kas02}. 
Every good module is a simple
$\uqpg$-module and a tensor product of a good module is again a good module.
 For example, the fundamental representation
$V(\varpi_i)$ is a good $\uqpg$-module.

For a module $M$  in $\CC_\g$, let us denote the right and the left
dual of $M$ by $\rd M$ and $M^*$, respectively. 
That is, we have isomorphisms
\eq
&&\ba{l}
\Hom_{\uqpg}(M\tens X,Y)\simeq \Hom_{\uqpg}(X, \rd M\tens Y),\\[1ex]
\Hom_{\uqpg}(X\tens \rd M,Y)\simeq \Hom_{\uqpg}(X, Y\tens M),\\[1ex]
\Hom_{\uqpg}(M^*\tens X,Y)\simeq \Hom_{\uqpg}(X, M\tens Y),\\[1ex]
\Hom_{\uqpg}(X\tens M,Y)\simeq \Hom_{\uqpg}(X, Y\tens M^*)
\ea\label{eq:dual}
\eneq
which are 
functorial in $\uqpg$-modules $X$, $Y$.
Then we have 
$\rd\bl M_x\br=\bl \rd M\br_x$
and $\bl M_x\br^*=\bl M^*\br_x$ for any $x \in \cor^\times$. 
The dual of fundamental representations are as follows:
\eq
&&\rd V(\varpi_i) \simeq V(\varpi_{i^*})_{\pstar}, \qquad 
\bl V(\varpi_i)\br^* \simeq V(\varpi_{i^*})_{\pstar^{-1}} 
\eneq
for $i \in I_0$.
 Here, 
$\pstar=(-1)^{\ang{\rho^\vee,\delta}}q^{\ang{c,\rho}}$, and
$\rho$ (respectively, $\rho^\vee$) denotes an element in $P$ (respectively, $P^\vee$) such that 
$\lan h_i, \rho\ran=1$ (respectively $\lan \rho^\vee, \alpha_i \ran =1$) for every $i \in I$. The map
$i \mapsto i^*$ is the involution on $I_0$ determined by 
$\alpha_{i^*}=-w_0 \alpha_i$,
where $w_0$ denotes the longest element of $W_0=\langle s_i \, | \, i\in I_0 \rangle \subset W$.

\subsection{R-matrices for quantum affine algebras} We recall the notion of $R$-matrices for quantum affine algebras, For details see \cite[\S\;8]{Kas02}.
Let us choose the following {\em universal $R$-matrix}.
Let us take a basis $\{P_\nu\}_\nu$ of $U_q^+(\g)$
and a basis $\{Q_\nu\}_\nu$ of $U_q^-(\g)$
dual to each other with respect to a suitable coupling between
$U_q^+(\g)$ and $U_q^-(\g)$.
Then for $\uqpg$-modules $M$ and $N$
define
\eq\label{def:univ}
R_{MN}^\univ(u\otimes v)=q^{(\wt(u),\wt(v))}
\sum_\nu P_\nu v\otimes Q_\nu u\,,
\eneq
so that
$R_{MN}^\univ$ gives a $\uqpg$-linear homomorphism
from $M\otimes N$ to $N\otimes M$
provided that the infinite sum has a meaning.

Let $M$ and $N$ be $\uqpg$-modules in $\CC_\g$,
and let $z_1$ and $z_2$ be indeterminates.
Then $R_{M_{z_1},N_{z_2}}^\univ$
converges in the $(z_2/z_1)$-adic topology.
Hence we obtain a morphism of
$\cor[[z_2/z_1]]\tens_{\cor[z_2/z_1]}\cor[z_1^{\pm1},z_2^{\pm1}]\tens\uqpg$-modules
$$R_{M_{z_1},N_{z_2}}^\univ\col \cor[[z_2/z_1]]\tens_{\cor[z_2/z_1]}
(M_{z_1}\tens N_{z_2})
\to \cor[[z_2/z_1]]\tens_{\cor[z_2/z_1]}
(N_{z_2}\tens M_{z_1}).
$$
If there exist $a\in\cor((z_2/z_1))$ and
 a $\cor[z_1^{\pm1},z_2^{\pm1}]\tens\uqpg$-linear homomorphism
$$\Rren_{M_{z_1},N_{z_2}} \col M_{z_1}\tens N_{z_2}\to N_{z_2}\tens M_{z_1}$$
such that $R^\univ_{M_{z_1},N_{z_2}}=a \Rren_{M_{z_1},N_{z_2}}$,
then we say that $R^\univ_{M_{z_1},N_{z_2}}$ is {\em rationally renormalizable.}

Now assume further that $M$ and $N$ are non-zero.
Then, we can choose $\Rren_{M_{z_1},N_{z_2}}$ so that, for any $c_1,c_2\in \cor^\times$,
the specialization of $\Rren_{M_{z_1},N_{z_2}}$ at $z_1=c_1$, $z_2=c_2$
$$\Rren_{M_{z_1},N_{z_2}}\vert_{z_1=c_1,z_2=c_2}\col M_{c_1}\tens N_{c_2}\to N_{c_2}\tens M_{c_1}$$
does not vanish.
Such an $\Rren$ is unique up to a multiple of
$\cor[(z_2/z_1)^{\pm1}]^\times=\bigsqcup_{n\in\Z}\,\cor^\times (z_2/z_1)^{n}$, 
and it is called a {\em renormalized $R$-matrix}.
We write
$$\rmat{M,N}\seteq \Rren\vert_{z_1=z_2=1}\col M\tens N\to N\tens M,$$
and call it the {\em $R$-matrix} between $M$ and $N$.
The $R$-matrix $\rmat{M,N}$
is well defined up to a constant multiple
when $R^\univ_{M_{z_1},N_{z_2}}$ is rationally renormalizable.
By the definition, $\rmat{M,N}$ never vanishes.

\medskip
Now assume that $M_1$ and $M_2$ are simple $\uqpg$-modules in $\CC_\g$.
Then, the universal $R$-matrix
$R^\univ_{(M_1)_{z_1},(M_2)_{z_2}}$ is rationally renormalizable.
More precisely, we have the following.
Let $u_1$ and $u_2$
be dominant extremal weight vectors of $M_1$ and $M_2$,
respectively.
Then there exists $a_{M_1. M_2}(z_2/z_1)\in\cor[[z_2/z_1]]^\times$
such that 
$$R^\univ_{(M_1)_{z_1},(M_2)_{z_2}}\bl(u_1)_{z_1}\tens (u_2)_{z_2}\br=
a_{M_1. M_2}(z_2/z_1)\bl(u_2)_{z_2}\tens(u_1)_{z_1}\br.$$
Then $\Rnorm_{(M_1)_{z_1},(M_2)_{z_2}}\seteq a_{M_1. M_2}(z_2/z_1)^{-1}R^\univ_{(M_1)_{z_1},(M_2)_{z_2}}$
is a unique $\cor(z_1,z_2)\tens\uqpg$-module homomorphism
\begin{equation}\ba{l}
\Rnorm_{ (M_1)_{z_1},(M_2)_{z_2}} \col
\cor(z_1,z_2)\otimes_{\cor[z_1^{\pm1},z_2^{\pm1}]} \bl(M_1)_{z_1} \otimes (M_2)_{z_2}\br \\[2ex]
\hs{30ex}\To\cor(z_1,z_2)\otimes_{\cor[z_1^{\pm1},z_2^{\pm1}]}
 \bl(M_2)_{z_2} \otimes (M_1)_{z_1} \br
\ea\end{equation}
satisfying
\begin{equation}\Rnorm_{(M_1)_{z_1},(M_2)_{z_2}}\bl(u_1)_{z_1} \otimes (u_2)_{z_2}\br =
(u_2)_{z_2}\otimes (u_1)_{z_1}.
\end{equation}

Note that
$\cor(z_1,z_2)\otimes_{\cor[z_1^{\pm1},z_2^{\pm1}]}
 \big((M_1)_{z_1} \otimes (M_2)_{z_2} \big)$ is a simple
$\cor(z_1,z_2)\tens\uqpg$-module (\cite[Proposition 9.5]{Kas02}).
We call $\Rnorm_{(M_1)_{z_1},(M_2)_{z_2}}$ the
{\em normalized $R$-matrix}.

Let $d_{M_1,M_2}(u) \in \cor[u]$ be a monic polynomial of the
smallest degree such that the image of $d_{M_1,M_2}(z_2/z_1)
\Rnorm_{(M_1)_{z_1},(M_2)_{z_2}}$ is contained in $(M_2)_{z_2} \otimes(M_1)_{z_1}$. We call $d_{M_1,M_2}(u)$ the {\em denominator of
$\Rnorm_{M_1, M_2}$}.
Then,
\begin{equation}
\Rren_{M_{z_1},N_{z_2}}\seteq d_{M_1,M_2}(z_2/z_1)\Rnorm_{(M_1)_{z_1},(M_2)_{z_2}}
\col (M_1)_{z_1} \otimes (M_2)_{z_2} \To
(M_2)_{z_2} \otimes (M_1)_{z_1}
\end{equation}
is a renormalized $R$-matrix, and 
$$\rmat{M_1,M_2}\col M_1\tens M_2\To M_2\tens M_1$$
is equal to the specialization of
$\Rren_{M_{z_1},N_{z_2}}$ at $z_1=z_2=1$
up to a constant multiple.

\smallskip
A simple module $V \in \CC_\g$ is called \emph{real} if $V\tens V$ is simple.
\Lemma[\cite{KKKO14}] \label{lem:commute_equiv_qa}
Let $V,W$ be simple modules in $\CC_\g$ and assume that one of them is real. Then  
\bnum
  \item $V \tens W$ and $W \tens V$ have simple socles and simple heads.
  \item Moreover, $\Im(\rmat{V,W})$ is equal to the head of $V \tens W$
and socle of $W \tens V$.
\end{enumerate}
\enlemma

Similarly to the quiver Hecke algebra case,
for $V, W\in \Mod(\uqpg)$, 
we denote by $V \hconv W$ and $V \sconv W$ the head and the socle
of $V \tens W$, respectively. 

The following lemma can be
proved similarly to the quiver Hecke algebra case
(\cite[Proposition 3.2.9]{KKKO15mm}),
and we do not repeat the proof. 

\Lemma \label{lem:runi}
Let $V,W$ be simple modules in $\CC_g$ and assume that one of them is real.
Then one has
$$\Hom(V\tens W,W\tens V)=\cor \,\rmat{V,W}.$$
\enlemma

\subsection{Hernandez-Leclerc's subcategory} \label{subsec:HL category}
For each quantum affine algebra $\uqpg$, we define a quiver $\mathscr S(\g)$ as follows:
\eq&&
\parbox{70ex}{\be[{(1)}]
\item we take  the set of equivalence classes
$\hat I_{\g} \seteq (I_0 \times \cor^\times) /  \sim$ as the set of vertices,
where the equivalence relation is given by $(i,x) \sim (j,y)$
 if and only if $V(\varpi_i)_x \cong V(\varpi_j)_y$,

\item  we put $d$ many arrows from $(i,x)$ to $(j,y)$,
where $d$ denotes the order of zero of
$d_{V(\varpi_i),V(\varpi_j)}(z_{V(\varpi_j)} / z_{V(\varpi_i)})$ at
$z_{V(\varpi_j)} / z_{V(\varpi_i)} = {y / x}$. \ee} \eneq
Note that  $(i,x)$ and $(j,y)$ are linked by
at least one arrow in $\mathscr S(\g)$  if and only if
the tensor product $V(\varpi_i)_x \tensor V(\varpi_j)_y$ is reducible (\cite[Corollary 2.4]{AK}).

Let $\mathscr S_0(\g)$ be a connected component of $\mathscr S(\g)$.
Note that a connected component of $\mathscr S(\g)$ is unique up to a spectral parameter shift 
and hence $\mathscr S_0(\g)$ is uniquely determined up to a quiver isomorphism.
 Let $\CC^0_\g$ be the smallest full subcategory of $\CC_\g$ stable
under taking subquotients, extensions, tensor products and
containing $\set{V(\varpi_i)_x}{(i,x) \in \mathscr S_0(\g)}$.
This category for  symmetric affine type $\g$ was introduced in \cite{HL10}.
Note that every simple modules in $\CC_\g$ is a tensor product of certain parameter shifts of some simple modules in $\CC_\g^0$ (\cite[Section 3.7]{HL10}). 
The  Grothendieck ring $K(\CC^0_\g)$  of $\CC^0_\g$ is the polynomial ring generated by 
the classes of modules in $\set{V(\varpi_i)_x}{(i,x) \in \mathscr S_0(\g)}$ (\cite{FR99}).

\subsection{KLR-type quantum affine Schur-Weyl duality functors} \label{sec:SWfunctor}
In this subsection, we recall the  construction of the
generalized quantum affine Schur-Weyl duality functor
(\cite{KKK13A}).

Let $\uqpg$ be a quantum affine algebra over $\cor$.

Assume that we are given
an index set $J$, a family $\{V_j\}_{j\in J}$
of good $\uqpg$-modules and a map
 $X \colon J \rightarrow \cor^\times$.

We define a quiver $\Gamma^J$ associated with the datum $(J, X, \{V_j\}_{j\in J})$
 as follows:
\eq&&
\parbox{70ex}{\be[{(1)}]
\item we take $J$ as the set of vertices,
\item  we put $d_{ij}$ many arrows from $i$ to $j$,
where $d_{ij}$ denotes the order of zero  of
$d_{V_{i},V_{j}}(z_{V_{j}} / z_{V_{i}} )$ at $z_{V_{j}} / z_{V_{i}}= {X(j) / X(i)}$.
\ee}\label{gammaJ}
\eneq
Note that we have $d_{ij}d_{ji}=0$ for $i,j\in J$. 

We define a symmetric Cartan matrix $A^J =(a^J_{ij})_{i,j\in J}$  by
\eq \label{eq:Cartan matrix}
a^J_{ij}=\begin{cases}
2&\text{if $i=j$,}\\
-d_{ij}-d_{ji}&\text{if $i\not=j$.}\end{cases} \eneq

 We give a family of polynomials
$\{Q_{i,j}(u,v)\}_{i,j\in J}$ satisfying \eqref{eq:Q} with the form
$$Q_{i,j}(u,v)=\pm(u-v)^{-a_{ij}}\qtext{for $i\not=j$}$$
for some choices of sign $\pm$.

Then we  choose a family $\{P_{i,j}(u,v)\}_{i,j \in J}$
of elements in $\cor[[u,v]]$  satisfying the following conditions.
\eq&&\hs{3ex}\left\{
\parbox{71ex}{\be[{(1)}]
\item $P_{i,i}(u,v)=1$ for $i \in J$.
\item The homomorphism $P_{ij}(u,v) \Rnorm_{(V_i)_{X(i)},(V_j)_{X(j)}} (z_{(V_i)_{X(i)}}, z_{(V_j)_{X(j)}})$
has no  pole and no zero  at $u=v=0$, 
where   $\cor[z_{(V_i)_{X(i)}}^{\pm 1}, z_{(V_j)_{X(j)}}^{\pm 1}] \to\cor[[u,v]]$ is  given  by
 $z_{(V_i)_{X(i)}} \mapsto 1+u $ and  $z_{(V_j)_{X(j)}} \mapsto 1+v$.
\item  $ Q_{i,j}(u,v)=P_{i,j}(u,v)P_{j,i}(v,u)$ for $i\not=j$.
\ee}\right. \label{eq:Pij}
\eneq

We call such a family $\{P_{i,j}(u,v)\}_{i,j\in J}$ a {\em \dcf}. 
 Note that we have
$P_{i,j}(u,v)\in \cor[[u,v]]^\times(u-v)^{d_{ij}}$.

Let $\set{\alj_i}{i \in J}$  be the  set of simple roots
corresponding to the Cartan matrix $A^J$ and $\rootl^+_J=\sum_{i \in
J} \Z_{\ge 0}  \alj_i$ be the corresponding positive root lattice.
We define the coupling $(\ ,\ )$ by $(\alj_i,\alj_j)=a^J_{ij}$.

Let us denote by $R^{J}( \beta )$  $(\beta \in \rootl^+_J)$  the symmetric
\KLR\ associated with the Cartan matrix $A^J$ and the  parameter
$\{ Q_{i,j}(u,v)\}_{i,j\in J}$.

For each $\nu =(\nu_1,\ldots, \nu_n) \in J^{ \beta }$,
let
\begin{equation*}
\Oh_{\mathbb{T}^n, X(\nu)} = \cor [[X_1 - X(\nu_1), \ldots, X_n-X(\nu_n)]]
\end{equation*}
 be the completion of
the local ring $\sho_{\mathbb{T}^n, X(\nu)}$
 of $\mathbb{T}^n$ at  $X(\nu)\seteq(X(\nu_1),\ldots,X(\nu_n))$.
Here $\mathbb{T}=\on{Spec}(\cor[X,X^{-1}])$.
 Set
$$
V_\nu =( V_{\nu_1})_\aff \otimes \cdots \otimes
(V_{\nu_n})_\aff.$$ Then $V_{\nu}$ is a left
$\bl\cor[X_1^{\pm1},\ldots, X_n^{\pm1}]\otimes\uqpg \br$-module,
where $X_k=z_{V_{\nu_k}}$.
 We   define
\eqn
&&\hV_\nu \seteq \Oh_{\mathbb{T}^n, X(\nu)}\otimes _{\cor[X_1^{\pm1},\ldots, X_n^{\pm1}]}
V_\nu, \qquad
\bV[ \beta ] \seteq \soplus\nolimits_{\nu \in J^{ \beta }}
\hV_\nu e(\nu), \\
&&\quad \text{and} \quad \bVK[ \beta ] \seteq  
\soplus\nolimits_{\nu \in J^{ \beta }}\mathrm{Frac}\bl \Oh_{\mathbb{T}^n, X(\nu)} \br\otimes _{\cor[X_1^{\pm1},\ldots, X_n^{\pm1}]}
V_\nu,
\label{eq:Vhat}
\eneqn
where $\mathrm{Frac}\bl \Oh_{\mathbb{T}^n, X(\nu)}\br$ denotes the ring of fractions of $\Oh_{\mathbb{T}^n, X(\nu)}$.

 Let $e(\nu) \in R^J(\beta)$  act on 
$\bV[ \beta ]  $ as the projection and 
$x_k e(\nu) \in R^J(\beta)$ as the multiplication by
$X(\nu_k)^{-1}\bl X_k -X(\nu_k)\br e(\nu)$.

Assign  $e(\nu)\tau_a$ to a $\cor$-linear map on $\bVK[ \beta ]$ defined 
by 
\begin{equation*}
e(\nu) \tau_a = \begin{cases}
R^\nu_{a,a+1} \circ P_{\nu_a, \nu_{a+1}}(x_{a}, x_{a+1})\circ e(\nu) & \text{if $\nu_a \neq \nu_{a+1}$,} \\
(x_a -x_{a+1})^{-1}\circ (R^\nu_{a,a+1} -1)\circ e(\nu) & \text{if $\nu_a = \nu_{a+1}$,}
\end{cases}
\end{equation*}
where
$R^\nu_{a,a+1}: \bV[ \beta ] \to \bVK[\beta]$
is the $\cor$-linear map given by 
$$R^\nu_{a,a+1}  = \id \tens \cdots \id \tens \Rnorm_{V_{\nu_a},V_{\nu_{a+1}} } (z_{V_{\nu_a}},z_{V_{\nu_{a+1}}}) \tens \id \cdots \tens \id.$$

The following theorem is one of the main results of \cite{KKK13A}.

\begin{theorem}
The assignments above give a well-defined right action of  $ R^J(\beta)$ on  $\bV[\beta]$ 
which commutes with the left action of $\uqpg$.
\end{theorem}
 Note that $\bV[\beta]$ is understood to be the trivial $\uqpg$-module
$\cor$ when $\beta=0$. 

\begin{remark}
The action of $ R^J(\beta)$ on $\bV[\beta]$ 
depends on the choice of \dcfs.
In \cite{KKK13A}, we take 
$P_{i,j}(u,v) =(u-v)^{d_{ij}}$ 
as the \dcf.
However, the same proof still works for the above theorem
with an arbitrary choice of \dcfs.
\end{remark}

For each $\beta \in \rootl^+_J$, we define 
the functor
\begin{align}
\F_{ \beta } \colon \Mod(R^J(\beta)) &\rightarrow
\Mod(\uqpg)
\end{align}
 by
\eq \F_{ \beta } (M) \seteq \bV[{ \beta }] \label{eq:functor}
\otimes_{R^J({ \beta })} M, \label{eq:the functor}\eneq
where $M$ is an
$R^J(\beta)$-module.

Set \eqn \F\seteq \soplus_{\beta\in \rootl^+_J}\F_\beta
\col\soplus_{\beta\in \rootl^+_J} \Mod(R^J(\beta))
\rightarrow \Mod(\uqpg). \eneqn
 Note that the functor $\F$ depends on the choice of
the \dcf $\{P_{i,j}(u,v)\}_{i,j\in J}$.

\begin{theorem}[{\cite{KKK13A}}] \label{thm:exact}
If the Cartan matrix $A^J$ associated with  $R^J$  is of finite type
$A, D$ or $E$, then the functor  $\F_\beta$ is exact  for every $\beta
\in \rootl^+_J$.
\end{theorem}

For each $i \in J$, let $L(i)$ be the $1$-dimensional
$R^J(\alj_i)$-module  generated by a non-zero vector $u(i)$ with
relation $x_1 u(i) = 0$ and $e(j) u(i)=\delta(j=i) u(i)$ for $j \in
J$. 
Then the affinization $L(i)_z :=  \cor[z] \otimes L(i)$  is isomorphic to  $R^J(\alj_i)$
as a left $R^J(\alj_i)$-module.

By the construction, we have
\begin{prop} {\rm (}\cite[Proposition 3.2.2]{KKK13A}{\rm )} \label{prop:image of tau} \hfill
\bnum
\item For any $i \in J$,
we have
\begin{align}
  \mathcal F(L(i)_z) \simeq \cor[[z]]  \tens_{\cor[z_{(V_i)_{X(i)}}^{\pm 1}]} ((V_i)_{X(i)})_\aff, 
\end{align}
 where   $ \cor[z_{(V_i)_{X(i)}}^{\pm 1}] \to\cor[[z]]$ is  given  by
 $z_{(V_i)_{X(i)}} \mapsto 1+z $. 
In particular, we have
\eqn 
\F(L(i))\simeq (V_i)_{X(i)} \quad \text{for } i \in J.
\eneqn

\item For $i,j\in J$,
let   $$\phi=R_{L(i)_z,L(j)_{z'}}\col L(i)_z \conv L(j)_{z'}
\rightarrow L(j)_{z'} \conv L(i)_z.$$
That is, let $\phi$ be  the
  $R^J(\alj_i+\alj_j)$-module homomorphism given by
\begin{align}
  \phi\bl u(i)_z\otimes u(j)_{z'}\br = \vphi_1\bl u(j)_{z'} \otimes u(i)_z\br,
\end{align}
where $\vphi_1$ is the intertwiner in \eqref{def:int}. Then we have
\eq \label{eq:Fphi}&&\mathcal F(\phi)
=P_{i,j}(z,z')
\Rnorm_{\F(L(i)),\F(L(j))}  (z_{\F(L(i))}, z_{\F(L(j))})  \eneq
as a $U_q'(\g)$-module homomorphism
\ \eqn &&
 \cor[[z,z']]\tens_{\cor[z_{\F(L(i))}^{\pm 1},z_{\F(L(j))}^{\pm 1}]}
\bl \F(L(i))_\aff \otimes \F(L(j))_\aff\br \\*
&&\hs{10ex}\longrightarrow\cor[[z,z']]\tens_{\cor[z_{\F(L(i))}^{\pm 1},z_{\F(L(j))}^{\pm 1}]}
\bl \F(L(j))_\aff \otimes \F(L(i))_\aff\br 
\eneqn where   $\cor[z_{\F(L(i))}^{\pm 1},z_{\F(L(j))}^{\pm 1}] \to\cor[[z,z']]$ is  given  by
 $z_{\F(L(i))} \mapsto 1+z $ and  $z_{\F(L(j))} \mapsto 1+z' $.
\ee
\end{prop}

  Recall that $\CC_\g$ denotes the category of
finite-dimensional integrable $\uqpg$-modules.

\begin{theorem} \rm{(}\cite{KKK13A}\rm{)}  \label{thm:conv to tensor}
The functor $\F$ induces a tensor functor
$$\F\col\soplus_{\beta \in \rootl^+_J} R^J( \beta) \gmod
\to\CC_\g.$$ Namely, $\F$ sends finite-dimensional graded
$R^J(\beta)$-modules to $\uqpg$-modules in $\CC_\g$, and  there
exist canonical $\uqpg$-module isomorphisms
 \eqn&& \F(R^J(0))\simeq
\cor,  \quad  \F(M_1 \conv M_2) \simeq \F(M_1) \otimes \F(M_2) \eneqn
for $M_1 \in R^J( \beta_1 ) \gmod$ and $M_2 \in R^J( \beta_2 )\gmod$
such that the diagrams in  \cite[A.1.2]{KKK13A} are commutative . 
\end{theorem}

The following theorem can be proved in a similar way as in \cite[Theorem 4.1]{KP15} and we omit the proof.

\begin{theorem}
Assume that the functor $\F$ in {\rm Theorem \ref{thm:conv to tensor}} is exact. 
\bnum
\item For any simple module $M$ in $R^J( \beta) \gmod$, $\F(M)$ is either a simple module or a zero module. In particular, if $M$ is a real simple module and $\F(M)$ is non-zero, then $\F(M)$ is a real simple module.
\item Let $M$ and $N$ be simple modules  in $R^J \gmod$, and assume that one of them is real. Then $\F(M \hconv N )$ is zero or isomorphic to $\F(M)\hconv \F(N)$.
\end{enumerate}
\end{theorem}
Hence if the functor  $\F$ is exact, then $\F$ induces a surjective map from
a subset of the isomorphism classes of simple subquotients of a module $M \in R^J\gmod$  to the set of the isomorphism classes of
 simple subquotients of $\F(M)$.  Hence we have the following corollary.
\begin{corollary} \label{cor:subquotients}
Assume that the functor $\F$ in {\rm Theorem \ref{thm:conv to tensor}}
is exact and assume that $M_1,M_2,\ldots,M_r$ are modules in  $R^J\gmod$.
Then, for any simple subquotient $V$ of $\F(M_1)\tens F(M_2)\tens \cdots \tens \F(M_r)$,
there exists a simple subquotient $L$ of $M_1\conv M_2 \conv \cdots \conv M_r$ such that 
$V \simeq \F(L)$.
\end{corollary}

Assume that the functor $\F$ in Theorem \ref{thm:conv to tensor} is exact. 
Let us denote by $\CC_J$ the full subcategory of  $\CC_\g$  consisting of $V$ such that
every composition factor of $V$ appears as a composition factor of a tensor product of modules of 
the form $(V_i)_{X(i)} \simeq \F(L(i))$ $(i \in J)$. 
In other words,
$\CC_J$ is the smallest
abelian subcategory of  $\CC_\g$ which 
contains  all $(V_i)_{X(i)}$'s,
and is  stable under taking submodules, quotients, extensions 
and tensor products. 
Hence we have
\eqn 
\F \col \soplus_{\beta \in \rootl^+_J} R^J( \beta) \gmod \to \CC_J \subset \CC_\g.
\eneqn

The following lemma will be used in the next section.

\begin{lemma} \label{lem:head}
Assume that $\F$ is exact.
Let $M$, $N\in R^J \gmod$ be simple modules, and assume that one of them is real.
We assume further that $\de(M,N)\le1$.
Assume that $\F(M),\F(N) \neq 0$ and $\F(M)\otimes \F(N)$ is not simple. 
Then the homomorphism $\F(\rmat{M,N})$ is equal to $\rmat{\F(M),\F(N)}$ up to a non-zero constant multiple,
and we have $\F(M\hconv N) \simeq \F(M)\hconv \F(N)$. In particular, it is
a simple module.
\end{lemma}
\begin{proof}
If $\de(M,N)=0$, then $\F(M)\otimes \F(N)$ is simple, and we need nothing to prove.
Assume that $\de(M,N)=1$.
Then we have exact sequences (ignoring grading shifts)
\eqn
0\to N\hconv M\to M\conv N\to M\hconv N\to0,\\
0\to M\hconv N\to N\conv M\to N\hconv M\to0
\eneqn
by Lemma~\ref{lem:linked}.
Applying the exact functor $\F$, we obtain an exact sequence
$$0\to \F(N\hconv M)\to \F(M)\tens \F(N)\to \F(M\hconv N)\to0.$$
Note that $\F(N\hconv M)$ and  $\F(M\hconv N)$ are simple or zero.
Since the length of $\F(M)\tens \F(N)$ is at least two by the assumption,
we conclude that $\F(N\hconv M)$ and  $\F(M\hconv N)$ are simple modules.
Hence $\F(\rmat{M,N})$ which is the composition
$\F(M)\tens \F(N)\epito \F(M\hconv N)\monoto \F(N)\tens \F(M)$,
does not vanish.
Hence it is equal to $\rmat{\F(M),\F(N))}$ up to a non-zero constant multiple
by Lemma~\ref{lem:runi}.
\end{proof}

\section{Quantum affine algebra of type B and   duality functors}\label{sec:F2}
In this section, we will construct and study a generalized quantum affine Schur Weyl duality functor whose codomain is the module category of quantum affine algebra of type $B$.

\subsection{Quantum affine algebra of type B${}^{(1)}_n$}
We recall the quantum affine algebra of type $B^{(1)}_n$ and fix the conventions. 
  From now on, we fix $n \ge 2$ and set $N=2n$. 
Let $\g$ be the affine Kac-Moody algebra of type $B^{(1)}_n$.
The index set of simple roots is given by  $I=\{0,1,\ldots, n\}$ and we have $I_0=\{1,2,\ldots, n\}$.
The  Dynkin diagrams and the fundamental weights are 
 given as follows:
 
\eqn  \begin{tabular}{c|c|p{5cm}}
Type& Dynkin diagram &  Fundamental weights\\
\hline \hline
\raisebox{-2em}{$B_{n}^{(1)}$ ($n \ge 3$)}
 &
$
\xymatrix@R=3ex{ *{\circ}<3pt> \ar@{-}[dr]^<{0}   \\ 
&*{\circ}<3pt> \ar@{-}[r]_<{2}  & {} \ar@{.}[r]
& *{\circ}<3pt> \ar@{-}[r]_>{\,\,\,\ n-1} &*{\circ}<3pt>\ar@{=>}[r]_>{\,\,\,\,n} &*{\circ}<3pt>  \\
*{\circ}<3pt> \ar@{-}[ur]_<{1} }
$
  & 
\raisebox{-0.5em}{$\vp_1=\cl( \Lambda_1-\Lambda_0)$}  \newline
 \raisebox{-0.5em}{$\vp_i =\cl(\Lambda_i-2\Lambda_0) \ (2\le i \le n-1)$} \newline
\raisebox{-0.5em}{$\vp_n =\cl(\Lambda_n-\Lambda_0) $}
\\

\hline
\raisebox{-1em}{$B_{2}^{(1)}$} &
$$\xymatrix@R=3ex{*{\circ}<3pt>  \ar@{=>}[r]_<{\alpha_0} & *{\circ}<3pt> \ar@{<=}[r]_<{\alpha_2} \ar@{}[r]_>{\,\,\,\ \alpha_1}  & *{\circ}<3pt>}$$
 & $\vp_1=\cl( \Lambda_1-\Lambda_0)$  \newline
 $\vp_2 =\cl(\Lambda_2-\Lambda_0 )$ \\
\hline
\end{tabular}
\eneqn

Note that the Dynkin diagram of type $B^{(1)}_2$ in the table above
is denoted by $C^{(1)}_2$ in \cite{Kac}.

By \eqref{eq:normal},
We have 
\eqn 
(\alpha_i,\alpha_i )= \begin{cases}
2 &\text{if } \ i = 0,1,\ldots, n-1,\\
1 &\text{if } \ i = n.
\end{cases}
\eneqn
We set
\eq
q_s:=q_n=q^{1/2} \quad \text{and} \quad q_t:=(-1)^{n+1}q_s^{2n+1}.
\eneq

The null root and the canonical central element are given by
\eqn&&\ba{llllcl}
\delta&=&\al_0+\al_1+&2(\al_2+\cdots+\al_{n-1}&+&\al_n),\\
c&=&h_0+h_1+&2(h_2+\cdots+h_{n-1})&+&h_n.
\ea
\eneqn
It follows that
\eq
p^*=q^{2n-1}.
\eneq

The fundamental representations of $U_q'(B_n^{(1)})$ are of the form $V(\varpi_i)_x$ with  $1 \le i \le n$ and $x \in \cor^\times$.
For $k > n +1 $ or $k < 0$, $V(\varpi_{k})$ is understood to be zero, and the modules
$V(\varpi_{0})$ and $V(\varpi_{n+1})
$ are understood to be the trivial representation.

Recall that $d_{V(\vpi(i),V(\vpi_j)}(z)$ 
denotes  the denominator of the normalized R-matrix $\Rnorm_{V(\varpi_i),V(\varpi_j)_z}$.
The following formula of the denominators are given in \cite{Oh14}.

\begin{prop}[\cite{Oh14}] \label{prop:denominator}
For $1\leq k,l \leq n-1$, we have
\begin{align}
  d_{V(\varpi_k), V(\varpi_l)}(z)  =\prod_{s=1}^{\min (k,l)} (z-(-1)^{k+l}q^{|k-l|+2s})(z-(-1)^{k+l}q^{2n-k-l-1+2s}).
\end{align}

      For $1 \leq k \leq n-1$, we have
  \begin{align}
    d_{V(\varpi_k), V(\varpi_n)}(z) = \prod_{s=1}^{k}(z-(-1)^{n+k}q_s^{2n-2k-1+4s}).
  \end{align}

Finally,  we have
\begin{align}
  d_{V(\varpi_n), V(\varpi_n)}(z)&=\prod_{s=1}^{n} (z-(q_s)^{4s-2}).\label{eq:denonn}
\end{align}
\end{prop}
Note that $d_{V(\vpi_i),V(\vpi_j)}(z)=d_{V(\vpi_j),V(\vpi_i)}(z)$.

\begin{prop}[\cite{Oh14}] \label{prop:Dorey}
There exist surjective $U'_q(B_n^{(1)})$-module homomorphisms
\eq
\hs{3ex}&& V(\varpi_k)_{(-q)^{-l}} \otimes V(\varpi_l)_{(-q)^{k}}
  \epito V(\varpi_{k+l})\ 
\text{for $1 \leq k,l\leq n-1$ with $k+l\le n-1$,} 
\label{eq:Doreykl}\\[1ex]
  && V(\varpi_n)_{q^{k+1}} \otimes V(\varpi_n)_{q^{N-k}}\epito 
V(\varpi_{k})_{(-1)^{k+1}q_t}\ \text{for $1 \leq k \leq n-1$,}
\label{eq:Doreynn}\\[1ex] 
&&V(\varpi_n)\tens V(\varpi_1)_{\qt} \epito V(\varpi_n)_{q^2}.
 \label{eq:Doreyn1}
\eneq
\end{prop}

\subsection{The quiver $\Gamma$}  \label{subsec:quiver}

 Recall that  $\CC^0_{B_{n}^{(1)}}$ is the smallest full subcategory of $\CC_{B_{n}^{(1)}}$
stable under taking subquotients,  extensions, tensor
products and containing $\set{V(\varpi_i)_x}{(i,x) \in \mathscr
S_0({B_{n}^{(1)}})}$.
Here we  take the following connected component $\mathscr S_0({B_{n}^{(1)}})$ of $\mathscr S({B_{n}^{(1)}})$:
\eqn
\mathscr S_0({B_{n}^{(1)}}) \seteq \set{ (i,(-1)^{i-1}q_t q^m),  (n,q^m)}{1\le i \le n-1, \ m \in \Z}.
\eneqn

Let
$J=\Z$, and
define
\eqn
V_j:=\bc V(\vp_n) & \text{if } \ j\equiv -1,0 \ \bmod N, \\
                        V(\vp_1) & \text{otherwise}.
                    \ec
\eneqn
for $j\in J$.
The map $X\col J\to \cor^\times$ is given as follows:
\begin{eqnarray*}
&&X(0)=q^0=1, \\
&&X(j)=\qt q^{2(j-1)} \qquad(1\le j \le N-2), \\
&&X(N-1)=q^{3N-5}, \\
\end{eqnarray*}
and we extend it by
\eqn
&&
X(j+kN)\seteq X(j)q^{k(2N-2)}=X(j)\pstar^{2k} \qtext{for $0\le j\le N-1$ and $k\in\Z$.} 
\eneqn

Then by Proposition \ref{prop:denominator} we have 
\eq
d_{ij}=
\delta(j=i+1, j \not\equiv 0 \bmod N)+ \delta(j=i-1, j \equiv -1 \bmod N).
\eneq
For $i,j\in J$, we have 
\begin{align*}(\alj_i,\alj_j)=\begin{cases}-1&\text{if $i-j=\pm1$,}\\
2&\text{if $i=j$,}\\
0&\text{otherwise}.
\end{cases}
\end{align*}
Hence  the quiver $\Gamma=\Gamma^J$ corresponding to the datum 
$(J,X,\{V_j\}_{j\in J})$  and the corresponding \KLR\ are of type $A_\infty$.

 We give the parameters for $R^J$ by
$$Q_{i,j}(u,v)
=\begin{cases}
 \pm(u-v)&\text{if $j=i\pm1$,}   \\
0&\text{if $i=j$,}\\
1&\text{otherwise.}
\end{cases}$$

\begin{remark}
The above choice of  parameters $\{Q_{i,j}(u,v)\}_{ij \in J}$ for quiver Hecke algebra $R$ of type $A_\infty$ is identical  with  those in \cite{KKK13A}
and \cite{KKKO15}.
\end{remark}

Let us choose a \dcf
$\set{P_{i,j}(u,v) \in \cor[[u,v]]}{i,j \in J}$ satisfying
condition \eqref{eq:Pij} including
$$Q_{i,j}(u,v) = P_{i,j}(u,v)P_{j,i}(v,u)\qtext{for $i\not=j$.}$$

For  example,  we may take 
\eq
\Pt_{i,j}(u,v)\seteq (-1)^{\delta(j=i+1\equiv 0 \bmod N)}(u-v)^{d_{ij}}.
\label{eq:PT}
\eneq
as $P_{i,j}(u,v)$.

\medskip

We take
$$P_J=\soplus_{a\in \Z}\Z\oep_a$$
as the weight lattice with $(\oep_a,\oep_b)=\delta_{a,b}$. The root
lattice $\rtl_J= \soplus_{i\in J}\Z\alj_i$ is embedded into $P_J$ by
$\alj_i=\oep_i-\oep_{i+1}$.
We write  $\rtl_J^+$ for $\soplus_{i\in
J}\Z_{\ge0}\alj_i$.

Then we have a family of functors $\F_\beta$ ($\beta \in \rtl_J^+)$ defined in \eqref{eq:functor}.
 Note that $\F_\beta$ depends on the choice of  \dcf $\{P_{i,j}(u,v)\}_{i,j\in J}$.

By  Theorem \ref{thm:exact},
 the functor 
  $$\F =\soplus_{\beta\in \rootl^+_J}\F_\beta\col \soplus_{\beta\in  \rtl^+_J} \Modg(R^J(\beta))\to
\Mod( U_q'(B_n^{(1)})) $$
 is exact. 

  Note that 
$$\F(L(a+N))\simeq \rd \rd \F(L(a))\qtextq{and}  \F(L(a-N))\simeq \F(L(a))^{**}$$
for all $a\in J$ by the construction.

\bigskip

\subsection{Simple modules over quiver Hecke algebra of type A${}_\infty$}
Since we deal with a functor $\F$ whose domain is the category of modules over the quiver Hecke algebra of type $A_{\infty}$, 
we recall some basic materials for the quiver Hecke algebra of type $A_{\infty}$. 

For simplicity we will write $R(\beta)$ for $R^J(\beta)$ in the sequel. 
 A pair of integers $(a,b)$ such that $ a \leq b$
is called a {\em segment}.  The \emph{length} of $(a,b)$  is $b-a+1$.
A \emph{multisegment} is a  finite sequence of  segments.

 For a segment $(a,b)$ of length $\ell$, we define a graded $1$-dimensional
 $R( \oep_a-\oep_{b+1} )$-module
 $L(a,b)=\cor u{(a,b)}$ in $R( \oep_a-\oep_{b+1} )\gmod$ which is generated by a
vector $u{(a,b)}$ of degree $0$ with the action of $R( \oep_a-\oep_{b+1} )$ given by
 \begin{align}
x_m u{(a,b)} =0 , && \tau_k u{(a,b)} =0 , &&
e(\nu) u{(a,b)}= \begin{cases}
u{(a,b)} & \text{if $\nu=(a,a+1, \ldots, b)$,} \\
0 & \text{otherwise.}
\end{cases}
\end{align}
We understand that $L(a,a-1)$ is the 1-dimensional module over $R(0) = \cor$  and
the length of $(a,a-1)$ is $0$.
Note that $L(a,a)$ is nothing but $L(a)$.

We give a total order on the set of segments as follows:
\eqn \label{eq:right order}
(a_1,b_1) > (a_2,b_2) \quad  \ \text{if} \ a_1 > a_2 \ \text{or} \ a_1=a_2 \ \text{and} \ b_1  <   b_2.
\eneqn

Then we have
\begin{prop} [{\cite[Theorem 4.8, Theorem 5.1]{KP11},
\cite[Proposition 4.2.7]{KKK13A}}] \label{prop:multisegments}\hfill
\bnum
\item
Let $M$ be a simple module in   $R(\beta) \gmod$ with 
$\beta\in\rootl_J^+$ .
Then there exists a unique pair of a multisegment $\big ((a_1,b_1),
\ldots , (a_t,b_t) \big)$ and an integer $c$
such that
\bna
\item $(a_k, b_k) \ge (a_{k+1}, b_{k+1})$ for $1 \leq k \leq t-1$,
\item $\sum_{k=1}^{t}(\oep_{a_k}-\oep_{b_k+1})=\beta$,
\item
$ M \simeq q^c\hd \big(L(a_1,b_1)\conv \cdots \conv L(a_t,b_t) \big)$,
 where $\hd$ denotes the head.
\ee

\item Conversely, if  a multisegment   $\big ((a_1,b_1), \ldots , (a_t,b_t) \big)$  satisfies
$\rm (a)$ and $\rm (b)$, then
$\hd \big(L(a_1,b_1)\conv \cdots \conv L(a_t,b_t) \big)$
is a simple $R(\beta)$-module.
\ee
\end{prop}

If a multisegment $\big ((a_1,b_1), \ldots , (a_t,b_t) \big)$
satisfies the condition (a) above, then
we say that it is an {\em ordered multisegment}.
We call the  ordered  multisegment $\big ((a_k,b_k)\big)_{1\le k \le t}$
in Proposition \ref{prop:multisegments}~(i)
the {\em multisegment associated with $M$}.

\begin{prop} [{\cite[Proposition 4.2.3]{KKK13A} }]\label{prop:exact sequences}
For $a \leq b$ and $a' \leq b'$,
 set $\beta=\oep_{a}-\oep_{b+1}$
and $\beta'=\oep_{a'}-\oep_{b'+1}$.
  \bnum
\item $L(a,b)$ is a real simple module.
\item We have
\eqn
\Lambda(L(a,b),L(a',b'))
=
\begin{cases}
-\delta_{a,a'}-\delta_{b,b'}+2 & \text{if } \ a \le a'\le b\le b', \\
-(\beta,\beta') &\text{otherwise}.
\end{cases}
\eneqn

\item  We have
\eqn
&&\de\bl L(a,b),L(a',b')\br\\
&&\hs{5ex}=\bc0&\text{if $[a,b]\subset [a',b']$ or
$[a,b]\supset [a',b']$ or $[a-1,b+1]\cap[a',b']=\emptyset$,}\\
1&\text{otherwise.}
\ec\eneqn

\item If $a' < a\le b' < b$, then
we have the following exact sequence
\begin{align*}
 & 0 \To qL(a',b) \conv L(a, b') \To
  L(a, b) \conv L(a',b')\\
&\hs{15ex}\To[{\hs{1.5ex} \rmat{L(a,b),L(a'b')} \hs{1.5ex}}] L(a',b') \conv L(a, b)\To
 q^{-1}L(a', b) \conv L(a,b')\To 0.\end{align*}
\item If $a=b'+1$, then we have an exact sequence
\begin{align*}
  0 \to q L(a',b) \To
 L(a, b) \conv L(a',b')
\To[{\rmat{L(a,b),L(a'b')}}] q^{-1}L(a',b') \conv L(a, b)\to
q^{-1} L(a', b)\rightarrow 0.
\end{align*}
\ee
\end{prop}

\subsection{Properties of the functor $\F$}

In this subsection we will investigate properties of the exact functor 
$$\F \col \soplus_{\beta \in \rootl^+_J} R( \beta) \gmod \to \CC_J \subset 
\CC_{B^{(1)}_n}.$$ 

\begin{lemma} \label{lem:FK0k}
We have
\eq \label{eq:0k}
\hs{2em}\F(L(0,k)) \simeq 
\begin{cases}
V(\varpi_n)_{q^{2k}} & \text{for} \ 0 \le k \le N-2, \\
 \cor &\text{for } \ k=N-1,
\end{cases}
\eneq
and
\eq
\F(L(k,N-1)) \simeq V(\varpi_n)_{q^{N-3+2k}} \qtext{for} \ 1 \le k \le N-1.
\eneq
\end{lemma}
\begin{proof}
When $k=0$,  isomorphism \eqref{eq:0k} is obvious  from the definition. Let $k \ge 1$. 
By induction on $k$, we may assume that  $\F(L(0,k-1))\simeq V(\varpi_n)_{q^{2(k-1)}}$.

On the other hand, Proposition \ref{prop:denominator} implies that
\eqn
&&\F\bl L(0,k-1)\br\hconv \F\bl L(k)\br\\
&&\hs{10ex}\simeq 
\bc V(\varpi_n)_{q^{2(k-1)}} \hconv V(\varpi_1)_{\qt q^{2k-2}}
\simeq V(\varpi_n)_{q^{2k}} &\text{for $1 \le k \le N-2$,}\\
V(\varpi_n)_{q^{2(N-2)}} \hconv V(\varpi_n)_{q^{3N-5}} \simeq \cor
&\text{if $k=N-1$.}
\ec\eneqn
Hence the tensor product $\F(L(0,k-1)) \tens \F(L(k))$ is not simple for $1 \le k \le N-1$, and
Lemma \ref{lem:head} implies that 
$\F(L(0,k))\simeq\F(L(0,k-1)\hconv L(k)) \simeq \F(L(0,k-1)) \hconv \F(L(k))$.
Thus we obtain  \eqref{eq:0k}.

\medskip
Let $1 \le k \le N-1$.  Then, an epimorphism
\eqn
&&V(\varpi_n)_{q^{2k-2}} \tens \F(L(k,N-1)) \simeq \F(L(0,k-1)\conv L(k,N-1)) \\
&&{\hskip 10em}\epito \F(L(0,k-1)\hconv L(k,N-1)) \simeq \F(0,N-1) \simeq \cor, \eneqn
induces a non-zero homomorphism
$\F(L(k,N-1))\to \rd\bl V(\varpi_n)_{q^{2k-2}}\br$. 
Since $\F(L(k,N-1))$ is zero or simple, we obtain
\[ \F(L(k,N-1)) \simeq \rd\bl V(\varpi_n)_{q^{2k-2}} \br\simeq V(\varpi_n)_{q^{2k-2}\pstar} =V(\varpi_n)_{q^{N-3+2k}}. \qedhere \]
\end{proof}

Since $(V_j)_{X(j)}$ ($j \in J$)
belongs to 
 $\CC_{B_{n}^{(1)}}^0$,
the category $\CC_J$ is a full subcategory of $\CC_{B_{n}^{(1)}}^{0}$. 
The next proposition asserts that  they are actually the same.

\begin{prop}
 For any $(i,x)\in\mathscr S_0({B_{n}^{(1)}})$,
the fundamental representation $V(\varpi_i)_x$
belongs to $\CC_J$. 
In particular, we have $\CC_J=\CC_{B_{n}^{(1)}}^{0}$.
\end{prop}
\Proof
By Lemma \ref{lem:FK0k}, the fundamental representation 
$V(\varpi_n)_{q^m}$ is in the image of the functor $\F$
if $m\in2\Z$ and $0\le m\le 2N-4$ or if
$m\in 2\Z+1$ and $N-1\le m\le 3N-5$.

By the relation $X(a+k N) = X(a)q^{k(2N-2)}$,
the category $\CC_J$ contains the fundamental representations $V(\varpi_n)_{q^m}$ for all $m \in \Z$. 

By \eqref{eq:Doreynn} we have an epimorphism
\eqn V(\varpi_n)_{q^{1+i+m}} \tens V(\varpi_n)_{q^{N-i+m}} \epito V(\varpi_i)_{(-1)^{i+1}q_t q^m}\eneqn
for $1\le i \le n-1$ and $m \in \Z$.
Hence $\CC_J$ contains  the other fundamental representations  in $\CC_{B_{n}^{(1)}}^{(0)}$.
\QED

By  Corollary \ref{cor:subquotients}, we have the following
 \begin{corollary} \label{cor:CJ=C0g}
For any simple module $V$  in $\CC_{B_{n}^{(1)}}^0$, there exists a simple module $L \in R^J\gmod$  such that $\F(L)\simeq V$.
\end{corollary}

Thus we have
 \begin{corollary}  \label{cor:surjective}
The functor $\F$ induces a surjective ring homomorphism 
$$\phi_{\F} \col K(R^J\gmod) \epito K(\CC_{B_{n}^{(1)}}^0)$$
\end{corollary}

\Rem
Recall that a functor $F\col \shc\to\shc'$ is {\em essentially surjective}
if, for any $M'\in\shc'$, there exists $M\in\shc$ such that
$F(M)\simeq M'$. We do not know if
$\F\col R^J\gmod\to \CC_{B_{n}^{(1)}}^0$ is essentially surjective or not.
\enrem

\begin{prop} \label{prop:images}
Let $(a,b)$ be a segment with length $\ell:=b-a+1.$
 Then $\F(L(a,b))$  is non-zero if and only if $\ell \le N$. Moreover 
    we have
\begin{align*}
\F(L(a,b)) \simeq \cor \qtext{if $\ell=b-a+1 =N$}.
\end{align*}
\end{prop}
\begin{proof}
 In the course of
the proof, we omit the gradings. 

Assume that $\F(L(a,a+N-1)) \simeq \cor$ for some $a \in \Z$.
Then we have a homomorphism 
$$\F(L(a)) \tens \F(L(a+1,a+N-1)) \twoheadrightarrow \F(L(a,a+N-1)) \simeq  \cor.$$
Hence  $\F(L(a+1,a+N-1)$  is a right dual of $\F(L(a))$. 
It follows that $\F(L(a+N))\simeq \rd\rd \F(L(a))$ is a right dual of $\F(L(a+1,a+N-1)$. 
Since $L(a+1, a+N) = L(a+1,a+N-1) \hconv L(a+N)$, applying  Lemma \ref{lem:head}, 
we have 
$$\F(L(a+1, a+N))\simeq \F(L(a+1,a+N-1)) \hconv \F(L(a+N))\simeq \cor.$$
A similar argument shows that 
$$\F(L(a-1, a+N-2))\simeq \F(L(a-1)) \hconv \F(L(a+1,a+N-1)) \simeq \cor.$$
Combining it with \eqref{eq:0k} we conclude that $\F(L(a,a+N-1)) \simeq \cor$ for all $a \in \Z$.
\bigskip

Assume that $\ell=b-a+1 <N$. Then there exists a surjective homomorphism 
$$\F(L(a,b))\otimes \F(L(b+1,a+N-1)) \twoheadrightarrow \F(L(a,a+N-1))\simeq\cor $$
and hence  $\F(L(a,b)) \not\simeq 0$.

\bigskip 
Note that
$$\F(L(a))\simeq \F(L(a)) \tens \F(L(a+1,a+N-1))  \twoheadrightarrow \F(L(a,a+N))$$
and
$$\F(L(a+N))\simeq \F(L(a,a+N-1)) \tens \F(L(a+N)) \twoheadrightarrow \F(L(a,a+N))$$
for all $a \in \Z$.
 If $\F(L(a,a+N))$ is a non-zero module, then we have 
$$\F(L(a))\simeq \F(L(a,a+N)) \simeq \F(L(a+N))\simeq \F(L(a))_{q^{2N-2}},$$
which is a contradiction. Hence we conclude that $\F(L(a,a+N))\simeq0$ for all $a \in \Z$.
Now assume that $\ell=b-a+1 \ge N+2$. 
Then there exists a surjective homomorphism 
$$0 \simeq \F(L(a,a+N))\otimes \F(L(a+N+1,b)) \twoheadrightarrow \F(L(a,b))$$
and hence $\F(L(a,b)) \simeq 0$, as desired.
\QED

 \subsection{Quotients and localizations of  the category R-gmod}
In this subsection we will recall the quotient category and the localizations of $R\gmod$ introduced in \cite[\S 4.4 - \S 4.5]{KKK13A}.
For details of the constructions, see \cite[Appendix A and B]{KKK13A}.

Set $\As_{\beta} = R({\beta}) \gmod$ and set $\As =
\soplus_{{\beta} \in \rootl_J^+} \As_{\beta}$. 
Then we have  constructed  a
functor  $\F =\soplus_{{\beta} \in \rootl_J^+}
\F_\beta\col \As \rightarrow  \CC_J=\CC^0_{B^{(1)}_n} \subset \CC_{B^{(1)}_n}$.

Let $\Ss_N$ be the smallest Serre subcategory
of $\As$ such that
\eq&&
\parbox{70ex}{
\begin{enumerate}
\item $\Ss_N$ contains $L(a, a+N)$ for any $a\in\Z$,
\item $X \conv Y, \ Y \conv X \in \Ss_N$
for all $X \in \As$ and $Y \in \Ss_N$.
\end{enumerate}}
\eneq
Note that $\Ss_N$ contains $L(a,b)$ if $b\ge a+N$.

Let us denote by $\As/ \Ss_N$ the quotient category of $\As$ by $\Ss_N$
and denote by $\mathcal Q\col \As \rightarrow \As/ \Ss_N$ the canonical functor.
Since $\F$ sends $\Ss_N$ to $0$ by Proposition~\ref{prop:images}, 
the functor $\F \col \As \rightarrow
\CC_J$ factors through $\mathcal Q$, i.e., there exists  a functor
$\F'\col \As/\Ss_N\to \CC_J$  unique  up to an isomorphism
such that the diagram below  quasi-commutes.
\eqn&&
\xymatrix@C=8ex@R=4ex{\As\ar[r]^-{\mathcal{Q}}\ar[dr]_-{\F}
&\As/ \Ss_N\ar[d]^-{\F'}\\
& \CC_J .} 
\eneqn

\smallskip

Note that $\As$ and $\As/\Ss_N$ are tensor categories with the
convolution as tensor products. The module $R(0)\simeq \cor$ is a
unit object. Note also that $Q\seteq q R(0)$ is an invertible
central object of $\As/\Ss_N$ and $X\mapsto  Q\conv X\simeq X\conv
Q$ coincides with the grading shift functor.
 Moreover, the
functors $\mathcal Q$, $\F$ and $\F'$ are tensor
functors.

\smallskip

\bigskip
We will
recall the localizations  $\T'_N$ and $\T_N$ of $\As /\Ss_N$ introduced 
in \cite[\S4.5]{KKK13A} 
(where they are denoted by  $T'_J$ and $T_J$, respectively).

\begin{definition}
Let $S$ be the automorphism of $P_J\seteq\soplus_{a\in\Z}\Z\oep_a$
given by $S(\oep_a)=\oep_{a+N}$.
We define the bilinear form $B$ on $P_J$ by
\eq
&&B(x,y)=-\sum_{k>0}(S^kx,y)\qtext{for $x,y\in P_J$.}
\label{def:S}
\eneq
\end{definition}

\Def
We define the new tensor product
$\nconv\col\As/\Ss_N\times \As/\Ss_N\to\As/\Ss_N$ by
\eqn
X\nconv Y=q^{B(\al,\beta)} X\conv Y,
\eneqn
 where  $X\in(\As/\Ss_N)_{\al}$ and $Y\in(\As/\Ss_N)_{\beta}$.
\edf 
Then $\As/\Ss_N$ is endowed with a new
structure of tensor category by $\nconv$ as shown in
\cite[Appendix A.8]{KKK13A}.

Set
\begin{equation}
L_a\seteq L(a, a+N-1)
\quad\text{for} \  a \in \Z.
\label{eq;La}
\end{equation}

For $a,j\in\Z$, set
\eq \label{eq:faj}
&&f_{a,j}(z) \seteq (-1)^{\delta_{j,a+N}}z^{-\delta(a\le j<a+N-1)-\delta_{j,a+N}} \in \cor[z^{\pm1}].
\eneq

\begin{theorem}[{\cite[Theorem 4.5.8]{KKK13A}}] \label{thm:L_a commutes with X}
The following statements  hold.
\bnum
\item
$L_a$ is a \emph{central object} in  $\As/\Ss_N$; i.e., \bna
\item
$f_{a,j}(z)R_{L_a,L(j)_z}$ induces an isomorphism in $\As/\Ss_N$
$$R_a(X)\col L_a\nconv X\isoto X\nconv L_a$$
 functorial in  $X\in\As/\Ss_N$,
\item the diagram
$$\xymatrix@C=10ex
{L_a\nconv X\nconv Y\ar[r]^{R_a(X)\nconv Y}\ar[dr]_-{R_a(X\nconv Y)\hs{2ex} }&
X\nconv L_a\nconv Y\ar[d]^{X\nconv R_a(Y)}\\
& X\nconv Y \nconv L_a }
$$ is commutative in $\As/\Ss_N$ for any $X,Y\in\As/\Ss_N$.
\ee
\item The  isomorphism
$R_a(L_a)\col L_a\nconv L_a\isoto L_a\nconv L_a$
coincides with $\id_{L_a\nconv L_a}$ in $\As/\Ss_N$.
\item
For $a,b\in\Z$,
 the  isomorphisms
\eqn
 R_a(L_b)\col L_a\nconv L_b\isoto L_b\nconv L_a \ \text{and} \ R_b(L_a)\col L_b\nconv L_a\isoto
L_a\nconv L_b
\eneqn

 in $\As/\Ss_N$ are  inverse to each other.
\ee
\enth

\medskip

By the preceding theorem, $\{(L_a, R_a)\}_{a \in J}$
forms a \emph{commuting family
of central objects} in $(\AA / \Ss_N , \nconv)$
(See \cite[Appendix A. 4]{KKK13A}).
Following \cite[Appendix A. 6]{KKK13A}, we localize $(\AA / \Ss_N , \nconv)$ by this commuting family.
Let us denote by  $\T'_N$
the resulting category $(\AA / \Ss_N)[L_a^{\nconv -1}\mid a\in J]$ .
Let $\Upsilon \colon \AA / \Ss_N \to \T'_N$ be the  projection functor.
We denote by $\T_N$ the tensor category
$(\As / \Ss_N)\,[L_a\simeq\one\mid a\in J]$ and by $\Xi \col
\T'_N \to \T_N$ the canonical functor (see \cite[Appendix A.7]{KKK13A}
and \cite[Remark 4.5.9]{KKK13A}).
Thus we have a chain of tensor functors
$$\As \To[\ {\mathcal Q}\ ]\As/\Ss_N\To[\ \Upsilon\ ] \T'_N\seteq(\AA / \Ss_N)[L_a^{\nconv -1}\mid a\in J]
\To[\ \Xi\ ]
\T_N\seteq(\As / \Ss_N)[L_a\simeq\one\mid a\in J].$$

The categories $\T_N$ and $\T'_N$ are rigid tensor categories; i.e.,
every object has a right dual
and a left dual  (\cite[Theorem
4.6.3]{KKK13A}).

\medskip

Let us denote by  $\Irr(\T_N)$ the set of the isomorphism classes of simple objects in $\T_N$.
 Define an equivalence relation $\sim$ on $\Irr(\T_N)$ by $X \sim Y$ if and only if
$X \simeq q^c Y$ in $\T_N$ for some integer $c$.
We write $\Irr(\T_N)_{q=1}$
for $\Irr(\T_N) / \sim $.
\begin{prop}[{\cite[Proposition 4.7.3 (iii)]{KKK13A}}]
The set $\Irr(\T_N)_{q=1}$  is isomorphic 
to the set of ordered multisegments
$$\bl (a_1,b_1),\ldots,(a_r,b_r) \br$$ satisfying
\begin{equation} \label{eq:segments<N}
  b_k -a_k+1 <N \ \text{for any} \
    1 \le k \le r,
\end{equation}
by the correspondence
$$\bl (a_1,b_1),\ldots,(a_r,b_r) \br\mapsto \bl \Xi\circ \Upsilon
\circ{\mathcal Q} \br \bl \hd\bl L(a_1,b_1)\conv\cdots \conv L(a_r, b_r)\br\br.$$
\end{prop}

\bigskip

\subsection{Factoring through $\T_N$} \label{subsec:factoring}

Let us come back to the $B_n^{(1)}$-case.
Recall that $N=2n$.
We have constructed the tensor functor
$\F\col \As\To\CC_{B_{n}^{(1)}}^{0}$.  
In this subsection, we will show that the functor $\F$ 
factors through
the category  $\T_N$ under a {\em suitable} choice of 
\dcf $\{P_{i,j}(u,v)\}_{i,j \in J}$. 

First, we need  the following lemma.

\Lemma
 Set $V(a):=\F(L(a))$ for $a\in J$.  Recall that
$V(a)=V(\varpi_{i_a})_{X(a)}$ for some $i_a \in I_0$.
For $a,b\in J$, set
\eqn 
&& \Rnorm_{a,b}= (\Rnorm_{V(a),V(b)_{z_{V(b)}}}  \tens 1 \tens \cdots \tens 1) \circ   
(1 \tens \Rnorm_{V(a+1),V(b)} \tens 1 \tens \cdots \tens 1) \circ   
\cdots \\ 
&&\hskip 1em\circ (1\tens \cdots \tens 1 \tens  \Rnorm_{V(a+k), V(b)_{z_{V(b)}} }  \tens 1 \tens \cdots \tens 1)
\circ \cdots \circ  (1\tens \cdots \tens 1 \tens  \Rnorm_{V(a+N-1), V(b)_{z_{V(b)}}} )
\eneqn
and
$$g_{a,b}(z) \seteq \prod_{k=0}^{N-1} a_{V(\varpi_{i_{a+k}}),V(\varpi_{i_b})} (X(a+k)^{-1} X(b) (1+z))^{-1},$$ 
where $z=z_{V(b)}-1$ and $a_{M,N}(z_N)\in\cor[[z_N]]^\times$ denotes the ratio 
of $\Runiv_{M,N_{z_N}}$ and $\Rnorm_{M,N_{z_N}}$, i.e.,  
$\Runiv_{M,N_{z_N}}= a_{M,N}(z_N) \Rnorm_{M,N_{z_N}}$. 
Then the following diagram is commutative 
\eqn
\xymatrix{
\bl V(a) \tens \cdots \tens V(a+N-1) \br \tens V(b)_{z_{V(b)}} \ar[r]^{\Rnorm_{a,b}} 
\ar@{>>}[d]^{\vphi \tens V(b)_{z_{V(b)}} }
& V(b)_{z_{V(b)}} \tens  \bl V(a) \tens \cdots \tens V(a+N-1) \br 
\ar@{>>}[d]^{V(b)_{z_{V(b)}} \tens \vphi}  \\
\cor \tens V(b)_{z_{V(b)}}  \ar[d]_*[@]{\sim}& V(b)_{z_{V(b)}} \tens \cor \ar[d]_*[@]{\sim} \\
V(b)_{z_{V(b)}} \ar[r]^{g_{a,b}(z)} & V(b)_{z_{V(b)}}  \\
}
\eneqn
for any surjective  homomorphism $\vphi:  V(a) \tens \cdots \tens V(a+N-1)  \epito \cor$.
\end{lemma}
\Proof

Let
\eqn 
&& \Runiv= (\Runiv_{V(a),V(b)_{z_{V(b)}}} \tens 1 \tens \cdots \tens 1)\circ
(1 \tens \Runiv_{V(a+1),V(b)_{z_{V(b)}}} \tens 1 \tens \cdots \tens 1) \\  
&&\hs{40ex} \circ   
\cdots  
\circ  (1\tens \cdots \tens 1 \tens  \Runiv_{V(a+N-1),V(b)_{z_{V(b)}}}).
\eneqn

If $\Rnorm$ and  $g_{a,b}(z)$  in the diagram are replaced by $\Runiv$ and  the identity map respectively, then the square is commutative. 
Hence the lemma follows from $a_{M_x,N_y}(z)=a_{M,N}(x^{-1}yz)$ for any simple $\uqpg$ modules $M,N$ and $x,y \in \cor^\times$. 
\QED

We temporarily fix a \dcf $\{\Pt_{i,j}(u,v)\}_{i,j \in J}$,
e.g.\ as in \eqref{eq:PT},
and denote the corresponding functor from
$\As /\Ss_N$ to $\CC^{0}_{B_n^{(1)}}$ by $\Ft$. 
Recall $L_a= L(a, a+N-1)$ from \eqref{eq;La}
and $f_{a,b}(z)$ from \eqref{eq:faj}.
\begin{prop} \label{prop:hab}
Let 
\eq h_{a,b}(z)\seteq
\displaystyle f_{a,b}(z) g_{a,b}(z) \prod_{k=a}^{a+N-1} \Pt_{k,b}(0,z). 
\eneq
Then the following diagram is commutative
\eqn
\xymatrix@C+7em@R+0.5em{
\Ft(L_a \nconv L(b)_z)  \ar[r]^{\Ft(R_a(L(b)_z))} \ar[d]_*[@]{\sim} 
& \Ft(L(b)_z \nconv L_a) \ar[d]_*[@]{\sim} \\
\Ft(L_a) \tens \Ft(L(b)_z)    \ar[d]_*[@]{\sim}  ^{\vphi_a \tens \Ft(L(b)_z))}  
& \Ft(L(b)_z) \tens  \Ft(L_a) \ar[d]_*[@]{\sim}  ^{ \Ft(L(b)_z)) \tens \varphi_a}  \\
\cor \tens \Ft(L(b)_z) \ar[d]_*[@]{\sim}  
& \Ft(L(b)_z) \tens \cor \ar[d]_*[@]{\sim} \\
\Ft(L(b)_z) \ar[r]^{h_{a,b}(z)}  
& \Ft(L(b)_z)
}
\eneqn
for any isomorphism $\vphi_a : \Ft(L_a)  \isoto \cor$.
\end{prop}
\Proof
Combining  \eqref{eq:Fphi} with the above lemma, we have a commutative diagram
\eqn
\xymatrix@C+7em@R+0.5em{
\Ft(L_a \nconv L(b)_z)  \ar[r]^{\Ft(R_{ L_a,L(b)_z})} \ar[d]_*[@]{\sim}
 & \Ft(L(b)_z \nconv L_a) \ar[d]_*[@]{\sim} \\
\Ft(L(b)_z) \ar[r]^{g_{a,b}(z) \prod_{k=a}^{a+N-1} \Pt_{k,b}(0,z)}  & \Ft(L(b)_z),
}\eneqn
where the vertical isomorphisms are the compositions of the isomorphisms of the diagram in the proposition.
Since $R_a(L(b)_z) =f_{a,b}(z) R_{L_a,L(b)_z}$ by Theorem \ref{thm:L_a commutes with X} (i) (a),  we get the desired result.
\QED

\begin{corollary}
The function $h_{a,b}(z)$ has no poles and no zeros at $z=0$.
\end{corollary}
\Proof
 Since $R_a(L(b)_z)$ is an isomorphism,
$\Ft(R_a(L(b)_z))\simeq h_{a,b}(z) $ 
is a well-defined isomorphism. Hence  $h_{a,b}(z)$ has no pole and no zero  at $z=0$.
\QED

\begin{prop}
The following diagram is commutative
\eqn
\xymatrix@C+7em@R+0.5em{
\Ft(L_a \nconv L_b)  \ar[r]^{\Ft(R_a(L_b))} \ar[d]_*[@]{\sim} & \Ft(L_b \nconv L_a) \ar[d] _*[@]{\sim}\\
\Ft(L_a) \tens \Ft(L_b)   \ar[d]_*[@]{\sim}^{\vphi_a \tens \vphi_b} & \Ft(L_b)  \tens \Ft(L_a) \ar[d]_*[@]{\sim}^{\vphi_b \tens \vphi_a}  \\
\cor \tens \cor  \ar[d]_*[@]{\sim} & \cor \tens \cor  \ar[d]_*[@]{\sim} \\
\cor  \ar[r]^{ \prod_{k=b}^{b+N-1} h_{a,k}(0)}  & \cor
}
\eneqn
for any isomorphisms $\vphi_a : \Ft(L_a)  \isoto \cor$ and  $\vphi_b : \Ft(L_b)  \isoto \cor$. 
\end{prop}
\begin{proof}
Let 
$h$ be a composition of homomorphisms in $\As /\Ss_N$ as follows:
\eqn
&&h:=\bl (L(b)\nconv \cdots \nconv L(b+N-2))\nconv R_a(L(b+N-1))  \br
 \circ \cdots  \\
&& \circ \bl (L(b)\nconv \cdots \nconv L(b+k-1))\nconv R_a(L(b+k)) \nconv (L(b+k+1) \nconv \cdots \nconv L(b+N-1))  \br
 \circ \cdots  \\
&&  \circ \bl R_a(L(b)) \nconv (L(b+1)\nconv \cdots \nconv L(b+N-1)) \br.
\eneqn
Consider the following diagram
\eqn
\xymatrix@C+5em@R-0em{
\Ft(L_a \nconv L(b)\nconv \cdots \nconv L(b+N-1))  \ar[r]^{\Ft(h)} \ar@{>>}[d] 
& \Ft( L(b)\nconv \cdots \nconv L(b+N-1) \nconv L_a)  \ar@{>>}[d] \\
\Ft(L_a \nconv L_b)  \ar[r]^{\Ft(R_a(L_b))}  \ar[d]_*[@]{\sim} & \Ft(L_b\nconv L_a)  \ar[d]_*[@]{\sim} \\
\cor   \ar[r]^{\prod_{k=b}^{b+N-1}h_{a,k}(0)}&  \cor, 
}
\eneqn
where the vertical isomorphisms are the compositions of the isomorphisms of the diagram in the proposition.

The upper square is commutative. 
By  Proposition \ref{prop:hab}, 
the outer square is also commutative.
Hence we get the commutativity of the lower square, as desired.
\end{proof}

\begin{corollary} \label{cor:lab=prodhak0}
Set $\la_{a,b}:=\prod_{k=b}^{b+N-1}h_{a,k}(0)$.
Then we have
\eqn
\la_{a,a}=1 \quad (a \in \Z) \qtext{and} \quad \la_{a,b} \la_{b,a} =1 \quad (a,b \in \Z).
\eneqn
\end{corollary}
\begin{proof}
 By Theorem \ref{thm:L_a commutes with X} (ii),  we have $R_a(L_a) = \id_{L_a\nconv L_a}$ and hence 
 $\Ft(R_a(L_a))$ is the identity map on $\cor$. Hence $\la_{a,a}=1$ by the above proposition.
 Similarly, (iii) in Theorem \ref{thm:L_a commutes with X}  implies $\la_{a,b} \la_{b,a} =1$.
\QED

\Lemma \label{lem:lambdaab}
Let $\cor$ be a field and let $N$ be a positive integer.
Let $\{h_{a,b}\}_{a,b \in \Z}$ be a family of elements in $\cor^\times$.
Set
\eqn
\la_{a,b}:=\prod_{k=b}^{b+N-1} h_{a,k} \quad (a,b \in \Z).
\eneqn

Then the system of equations in $\{c_{a,b}\}_{a,b \in \Z}$
\eq \label{eq:cab}
&&\left\{\ba{l}\prod_{k=a}^{a+N-1} c_{k,b}=h_{a,b} \quad (a,b \in \Z), \\[1ex]
c_{a,a}=1\quad (a \in \Z)\qtext{and } \
c_{a,b} \ c_{b,a}=1 \quad (a,b \in \Z) 
\ea\right.
\eneq
has a solution 
if and only if
\eq \label{eq:hab}
\la_{a,a}=1\quad (a \in \Z)\qtext{and } \
\la_{a,b} \ \la_{b,a}=1 \quad (a,b \in \Z),
\eneq
\end{lemma}

\begin{proof}
The ``only if'' part is straightforward
because $\la_{a,b}=\prod_{k=a}^{a+N-1}\prod_{j=b}^{b+N-1} c_{k,j}$. 

\medskip
Assume that  \eqref{eq:hab}  holds. 
 By the construction, we have
\eq
 \dfrac{\la_{a,b+1}}{\la_{a,b}}=\dfrac{h_{a,b+N}}{h_{a,b}} 
\qtext{for any $a,b \in \Z$.}
\label{eq:lah}
\eneq

First, we take an  arbitrary $\{c_{a,b}\}_{0\le a,b \le N-2}$ satisfying 
$c_{a,a}=1$, $c_{a,b} \ c_{b,a}=1$. 

Next, we extend it to $\{c_{a,b}\}_{0\le a,b \le N-1}$ by
\eqn
&&c_{N-1,b}:=h_{0,b} \prod_{k=0}^{N-2} c_{k,b}^{-1} \quad (0 \le b \le N-2), \\*
&&c_{a,N-1}:= (c_{N-1,a})^{-1} \quad (0 \le a \le N-2), \\*
&&c_{N-1,N-1}=1.
\eneqn
Then we  define $c_{a,b}$ for all $a,b \in \Z$ 
 by extending  them  under  the conditions
\eq 
\label{eq:ca+Nb}&&c_{a+N,b}:=\dfrac{h_{a+1,b}}{h_{a,b}} c_{a,b}\ \qtext{and}  \\
\label{eq:cab+N}&&  c_{a,b+N}:=\dfrac{h_{b,a}}{h_{b+1,a}} c_{a,b} \qtext{for all } a,b \in \Z.
\eneq
Note that such an extension is possible because
$\vphi_{a,b}\seteq \dfrac{h_{a+1,b}}{h_{a,b}}$ and
$\psi_{a,b}\seteq\dfrac{h_{b,a}}{h_{b+1,a}}$ satisfy the cocycle condition
$$\vphi_{a,b+N}\psi_{a,b}=\psi_{a+N,b}\vphi_{a,b},$$
as seen in
\eqn
\dfrac{\vphi_{a,b+N}\psi_{a,b}}{\psi_{a+N,b}\vphi_{a,b}}
&=&\dfrac{h_{a+1,b+N}}{h_{a, b+N}}
\dfrac{h_{b,a}}{h_{b+1,a}}\dfrac{h_{b+1,a+N}}{h_{b,a+N}}
\dfrac{h_{a,b}}{h_{a+1, b}}
=\dfrac{h_{a+1,b+N}}{h_{a+1, b}}\dfrac{h_{a,b}}{h_{a, b+N}}
\dfrac{h_{b,a}}{h_{b,a+N}}\dfrac{h_{b+1,a+N}}{h_{b+1,a}}\\
&=&\dfrac{\la_{a+1,b+1}}{\la_{a+1, b}}\dfrac{\la_{a,b}}{h_{a, b+1}}
\dfrac{\la_{b,a}}{\la_{b,a+1}}\dfrac{\la_{b+1,a+1}}{\la_{b+1,a}}
=1.
\eneqn
Here the third equality follows from \eqref{eq:lah}.
By the construction, we can easily verify
\eq
c_{a,a}=1\ \text{for $0\le a\le N-1$ and}\ 
c_{a,b} \, c_{b,a}=1\ \text{for $0\le a,b\le N-1$.}
\label{eq:0N}
\eneq

\medskip
Then we have
\eqn
c_{a,b+N} \, c_{b+N,a} = \dfrac{h_{b,a}}{h_{b+1,a}}\dfrac{h_{b+1,a}}{h_{b,a}}\,c_{a,b}\, c_{b,a}=c_{a,b}\, c_{b,a},
\eneqn
and similarly
\eqn
c_{a+N,b} \, c_{b,a+N} = c_{a,b}\,c_{b,a}
\eneqn
for $a,b \in \Z$. 
It follows  from \eqref{eq:0N} that 
\eqn
c_{a,b}\, c_{b,a}=1\qtext{for all $a,b \in \Z$.}
\eneqn
\medskip

By \eqref{eq:lah}, we get
\eqn
c_{a+N,a+N} 
&&= \dfrac{h_{a+1,a+N}}{h_{a,a+N}}c_{a,a+N}=\dfrac{h_{a+1,a+N}}{h_{a,a+N}}\dfrac{h_{a,a}}{h_{a+1,a}} c_{a,a} \\
&&=\dfrac{h_{a+1,a+N}}{h_{a+1,a}} \dfrac{h_{a,a}}{h_{a,a+N}} c_{a,a}
=\dfrac{\la_{a+1,a+1}}{\la_{a+1,a}} \dfrac{\la_{a,a}}{\la_{a,a+1}} c_{a,a}=c_{a,a}
\eneqn
for all $a\in \Z$. Thus \eqref{eq:0N} implies
 $c_{a,a}=1$ for all $a\in \Z$.
\medskip

Note that by the definition we have
\eqn
\prod_{k=0}^{N-1}c_{k,b} =h_{0,b} \quad (0\le b \le N-2).
\eneqn
When $b=N-1$, we have
\eqn
\prod_{k=0}^{N-1}c_{k,N-1}=\prod_{k=0}^{N-1}c_{N-1,k}^{-1}=\prod_{k=0}^{N-2}c_{N-1,k}^{-1}
=\prod_{k=0}^{N-2}\bl h_{0,k}^{-1} \prod_{j=0}^{N-2} c_{j,k}\br.
\eneqn
Since $\prod_{0\le j,k\le N-2}c_{j,k}=1$, we have
$$\prod_{k=0}^{N-1}c_{k,N-1}=\prod_{k=0}^{N-2} h_{0,k}^{-1}
 = h_{0,N-1} \, \la_{0,0}^{-1}=h_{0,N-1}.$$ 
Hence we have  showed that 
$$\prod_{k=a}^{a+N-1} c_{k,b}=h_{a,b}$$ for  $a=0$ and $0 \le b\le N-1$.
Then by \eqref{eq:ca+Nb}, it holds for all $a \in \Z$ and $0 \le b\le N-1$.
Now by \eqref{eq:cab+N}, it holds for all $a \in \Z$ and $b \in \Z$, as desired.
\end{proof}


\begin{lemma} \label{lem:cabuv}
Let $\{h_{a,b}(z)\}_{a,b \in\Z} $ be a family of elements in  $\cor[[z]]$.
Assume that \eqn
\la_{a,b}:=\prod_{k=b}^{b+N-1} h_{a,k}(0) \quad (a,b \in \Z)
\eneqn
satisfies that 
\eqn
\la_{a,a}=1 \quad (a \in \Z) \qtext{and} \quad \la_{a,b} \la_{b,a} =1 \quad (a,b \in \Z).
\eneqn
Then there exists a  family $\{c_{a,b}(u,v)\}_{a,b \in\Z} $
of elements in $\cor[[u,v]]$
satisfying 
\eqn
&& c_{a,a}(u,v)=1 \quad  (a \in \Z), \quad c_{a,b}(u,v)\, c_{b,a}(v,u)=1 \quad (a,b \in \Z), \\
&&\text{and} \quad h_{a,b}(z)=\prod_{k=a}^{a+N-1} c_{k,b}(0,z) \quad (a,b \in \Z). 
\eneqn 
\end{lemma}
\Proof
By Lemma \ref{lem:lambdaab}, there exists a family
$\{c_{a,b}\}_{a,b \in\Z}$ of elements in  $\cor^\times$ satisfying
\eqn
&& c_{a,a}=1 \quad  (a \in \Z), \quad c_{a,b} c_{b,a}=1 \quad (a,b \in \Z), \\
&&\text{and} \quad h_{a,b}(0)=\prod_{k=a}^{a+N-1} c_{k,b} \quad (a,b \in \Z). 
\eneqn 

Set $\vphi_{a,a}(v)=1\in \cor[[v]]$ for $a \in \Z$. 
For $a,b \in \Z$  such that $a < b < a+N-1$, 
  we choose $\vphi_{a,b}(v) \in \cor[[v]]$   satisfying
$\vphi_{a,b}(0)=c_{a,b}$.
Then $\vphi_{a,b}(0)=c_{a,b}$ for  $a,b \in \Z$ with $a \le b < a+N-1$.
We extend them to all $a,b \in\Z$ inductively 
\eqn 
\vphi_{a,b}(v) := 
\begin{cases}
\dfrac{h_{a-N+1,b}(v)}{\prod_{k=a-N+1}^{a-1} \vphi_{k,b}(v)}, & {for}\ b < a, \\
\dfrac{h_{a,b}(v)}{\prod^{k=a+N-1}_{a+1} \vphi_{k,b}(v)}, & {for}\ b \ge a+N-1.
\end{cases}
\eneqn
Then we have
\eqn 
\vphi_{a,b}(0)=c_{a,b} \qtext{for} \ a,b \in\Z, 
\qtext{and} \ \prod_{k=a}^{a+N-1} \vphi_{k,b}(v) = h_{a,b}(v) \qtext{for} \ a,b \in \Z.
\eneqn
Define 
\eqn c_{a,b}(u,v):=\dfrac{\vphi_{a,b}(v)}{\vphi_{b,a}(u)} c_{b, a} \in \cor[[u,v]] \qtext{for} \ a,b \in\Z.
\eneqn
Then  $c_{a,a}(u,v)=1$ holds trivially, and we have 
\eqn
c_{a,b}(u,v)\,c_{b,a}(v,u) =  \dfrac{\vphi_{a,b}(v)}{\vphi_{b,a}(u)}
\dfrac{\vphi_{b,a}(u)}{\vphi_{a,b}(v)}c_{b,a} c_{a, b} = 1.
\eneqn 
We also have
\eqn 
c_{a,b}(0,z)= \dfrac{\vphi_{a,b}(z)}{\vphi_{b,a}(0)} c_{b, a} = \vphi_{a,b}(z)
\eneqn
and hence
\eqn
 h_{a,b}(z)=\prod_{k=a}^{a+N-1} \vphi_{k,b}(z)=
\prod_{k=a}^{a+N-1} c_{k,b}(0,z) \qtext{for} \ a, b \in \Z,
\eneqn
as desired.
\QED

\medskip
Let 
$\{c_{a,b}(u,v)\}_{a,b \in \Z}$  be a family obtained by applying Lemma \ref{lem:cabuv} to the family $\{h_{a,b}(z)\}_{a,b \in \Z}$ in \eqref{prop:hab}.
Note that  Corollary~\ref{cor:lab=prodhak0}
asserts that
the condition 
for $\{h_{a,b}(z)\}_{a,b \in \Z}$ in Lemma \ref{lem:cabuv} holds.

We now give a new  \dcf
 by
\eqn 
P_{a,b}^{\rm{new}}(u,v):=c_{a,b}(u,v)^{-1} \Pt_{a,b}(u,v)=c_{b,a}(v,u) \Pt_{a,b}(u,v),
\eneqn
and denote the corresponding functor from $\As/\Ss_N$ to $\CC_{B_n^{(1)}}$ by $\F'$.
Note that  the new \dcf $\{P_{a,b}^{\rm{new}}(u,v)\}_{a,b\in \Z}$ also satisfies the conditions in \eqref{eq:Pij}
including the equality
$$Q_{a,b}(u,v)=\delta(a\neq b)P_{a,b}^{\rm{new}}(u,v) P_{b,a}^{\rm{new}}(v,u)$$
since $c_{a,b}(u,v) c_{b,a}(v,u)=1$. 

\Th \label{thm:cij}
Let $a\in J$ and $M \in \As/\Ss_N$.  
The the diagram
\eqn &&\hs{-4ex}
\ba{l}\xymatrix@C=8.5ex{
\F'(L_a\nconv M)\ar[d]^-{\F'(R_a(M))}\ar[r]^-\sim
&\F'(L_a)\tens\F'(M)\ar[r]^-{ g_a \tens \F'(M)}
&\cor\tens \F'(M)\ar[dr]\\
\F'(M\nconv L_a)\ar[r]^-\sim&\F'(M)\tens\F'(L_a)\ar[r]^-{\F'(M)\tens
g_a} &\F'(M)\tens  \cor  \ar[r]&\F'(M) }\ea\label{eq:1com} \eneqn
 is commutative  for any isomorphism $ g_a \col\F'(L_a)\isoto\cor$.
\enth

\Proof
It is enough to show the commutativity of the diagram in the cases $M=L(b)_z$ for $b \in J$.
By the same reasoning in Proposition \ref{prop:hab}, we need to show
$f_{a,b}(z) g_{a,b}(z) \prod_{k=a}^{a+N-1} P^{\rm{new}}_{k,b}(0,z)=1$.
This follows from
$$
\displaystyle f_{a,b}(z) g_{a,b}(z) \prod_{k=a}^{a+N-1} P^{\rm{new}}_{k,b}(0,z) 
=h_{a,b}(z) \prod_{k=a}^{a+N-1} c_{k,b}(0,z)^{-1} =1,$$
where  the last  equality is a consequence of  Lemma \ref{lem:cabuv}.
\QED

Then \cite[Proposition A.7.3]{KKK13A} and  \cite[Proposition A.7.2]{KKK13A}  imply
\begin{theorem}
There exists an exact tensor functor $\tF \col \T_N\to \CC^0_{{B_{n}^{(1)}}}$
such that the following diagram quasi-commutes:
\begin{equation}
 \xymatrix@C+7ex{
\As\ar[r]^-{\mathcal Q}\ar[drr]_-{\F}&\As/\Ss_N\ar[r]^{\Upsilon }\ar[rd]^(.55){\F'}
&\T'_N\ar[d]\ar[r]^{ \Xi  }&\T_N\ar[dl]^-{\tF}\\
&&\CC^0_{{B_{n}^{(1)}}}\,.}
\end{equation}
\end{theorem}

Hence $\tF$ induces a surjective ring homomorphism $\phi_{\tF} \col K(\T_N)_{q=1} \epito K(\CC^0_{{B_{n}^{(1)}}})$, where 
$K(\T_N)_{q=1}\seteq K(\T_N)/(q-1)K(\T_N)$. 

\begin{corollary}\label{cor:tFs}
The functor $\tF$ sends a non-zero object in $\T_N$ to a non-zero module in $\CC^0_{{B_{n}^{(1)}}}$.
In particular, it sends a simple to a simple.
\end{corollary}
\Proof
Let $M$ be a non-zero object in $\T_N$. Since the category $\T_N$ is rigid, 
there exists $M^* \in \T_N$ such that there is an epimorphism 
$M^* \nconv M \epito \cor.$
 Since $\tF$ is exact, we have a surjective homomorphism
$\tF(M^*) \tens \tF(M) \epito \cor$. In particular, $\tF(M)$ is non-zero.
\QED

In the rest of this subsection, we will show that
 $\phi_{\tF}$ is an isomorphism and  induces a bijection between the classes of simple objects. 

\Lemma \label{lem:cancel}
Let  $S$ and $S'$ be simple modules in $\As$ with
$S \simeq \hd(L(a_1,b_1) \conv \cdots\conv L(a_r,b_r))$ and 
$S\rq{} \simeq \hd(L(a'_1,b'_1) \conv \cdots\conv L(a'_t,b'_t))$ 
where
$\bl(a_1,b_1),\ldots,(a_r,b_r)\br$ and $\bl(a'_1,b'_1),\ldots,(a'_t,b'_t)\br$
are ordered multisegments such that 
$b_i-a_i+1 \le N$ $(1 \le i \le r)$ and $b'_j-a'_j+1 \le N$ $(1 \le j \le t)$.
Assume that $a_1=a_1'$ and $\F(S) \simeq \F(S')$. 
Then we have $$\F(S_0)\simeq \F(S'_0)\not\simeq0,$$
where
$S_0=\hd(L(a_1+1,b_1) \conv \cdots \conv L(a_r,b_r))$  and $S'_0=\hd(L(a'_1+1,b'_1)\conv \cdots\conv L(a'_t,b'_t))$. 
\enlemma
\Proof  By the condition on the lengths of segments,
$S$ and $S'$ are simple as objects of $\T_N$.
Hence $\F(S) \simeq \F(S')$ is a simple module by Corollary~\ref{cor:tFs}. 
By  \cite[Theorem 3.2, Corollary 6.1]{V02}, we have 
\eqn 
S\simeq L(a_1)\hconv S_0 \qtext{and} \ S' \simeq L(a_1)\hconv S'_0, 
\eneqn
which implies that there exists an epimorphism
$\F(L(a_1))\tens \F(S_0)\epito \F(S)$.
Hence $\F(S_0)$ is non-zero and it is a simple module.
Since $\F(S)$ is a simple module, we have
$ \F(L(a_1)) \hconv \F(S_0) \simeq \F(S)$
and similarly $\F(L(a_1)) \hconv \F(S'_0)\simeq\F(S')$.

They imply that $ \F(L(a_1)) \hconv \F(S_0) \simeq  \F(L(a_1)) \hconv \F(S'_0)$.

By \cite[Corollary 3.14]{KKKO14}, we have $\F(S_0) \simeq \F(S'_0)$, as desired.
\QED

\Prop \label{prop:injective}
Let $S$ and $S'$ be simple objects  in $\T_N$ with 
$$\tF(S) \simeq \tF(S'),$$
then we have $S\simeq S'$ in $\T_N$ up to a grading shift.
\enprop
\Proof
Let $$\bl(a_1,b_1),\ldots,(a_r,b_r)\br\qtext{and} \quad\bl(a'_1,b'_1),\ldots,(a'_t,b'_t)\br$$ 
be the multisegments associated with $S$ and $S'$, respectively. 
Then we have
$b_i-a_i+1<N$ $(1 \le i \le r)$ and $b'_j-a'_j+1 <N$ $(1 \le j \le t)$. 
We shall show that if $\tF(S) \simeq \tF(S')$, then we have $a_1=a'_1$.
Once we show it, then the proposition follows
from the above lemma
by induction on $\sum (b_i-a_i+1) + \sum (b'_j-a'_j+1)$.

Assume that $a_1> a'_1$. Then the multisegment 
$((a_1,a_1+N-1),  (a'_1,b'_1), \ldots , (a'_t,b'_t))$ is ordered. Hence 
we have
\eqn
&&\tF(\hd(L(a_1,a_1+N-1)\conv L(a'_1,b'_1)\conv\cdots \conv L(a'_t,b'_t)))
\\
&&\hskip 2em \simeq
\tF(S')\simeq  
\tF(S)
\simeq
\tF(\hd(L(a_1,b_1)\conv\cdots \conv L(a_r,b_r))).
\eneqn
Applying Lemma \ref{lem:cancel} ($b_1-a_1+1$) times,
we have
\eqn
&&\tF(\hd(L(b_1+1,a_1+N-1)\conv L(a'_1,b'_1)\conv\cdots \conv L(a'_t,b'_t))) \\
&&\hskip 2em\simeq
\tF(\hd(L(a_2, b_2)\conv\cdots \conv L(a_r,b_r))).
\eneqn
Note that 
$b_1+1 > a_1  \ge a_2$. Thus 
 we  have
\eqn
&&\tF(\hd(L(b_1+1,a_1+N-1)\conv L(a'_1,b'_1)\conv\cdots \conv L(a'_t,b'_t))) \\
&&\hskip 2em \simeq
\tF(\hd(L(b_1+1, b_1+N) \conv L(a_2, b_2)\conv\cdots \conv L(a_r,b_r))).
\eneqn
Now we apply  Lemma \ref{lem:cancel} again in order to obtain
\eqn
&&\tF(\hd(L(a'_1,b'_1)\conv\cdots \conv L(a'_t,b'_t))) \\
&&\hskip 2em\simeq
\tF(\hd(L(a_1+N, b_1+N) \conv L(a_2, b_2)\conv\cdots \conv L(a_r,b_r))).
\eneqn
Note that $a_1+N > a_1 > a_1'$. Hence, repeating the above procedure,  we get
\eqn
&&\tF(\hd(L(a'_1,b'_1)\conv\cdots \conv L(a'_t,b'_t))) \\
&&\hskip 2em\simeq
\tF(\hd(L(a_1+k N, b_1+ k N) \conv L(a_2, b_2)\conv\cdots \conv L(a_r,b_r)))
\eneqn
for all $k \in \Z_{\ge 0}.$

Note that if $k$ is sufficiently large, then we have 
\eqn
&&\hd(L(a_1+k N, b_1+ k N) \conv L(a_2, b_2)\conv\cdots \conv L(a_r,b_r)) \\
&&\hskip2em\simeq
L(a_1+k N, b_1+ k N) \conv \hd(L(a_2, b_2)\conv\cdots \conv L(a_r,b_r)).
\eneqn
Hence it  follows that 
\eqn
\tF(S') \simeq
\tF(L(a_1+k N, b_1+ k N)) \tens \tF\bl\hd(L(a_2, b_2)\conv\cdots \conv L(a_r,b_r))\br
\eneqn
for all sufficiently large $k$. 
Hence the class of $\tF\bl\hd(L(a_2, b_2)\conv\cdots \conv L(a_r,b_r))\br$ in the Grothendieck ring $K(\CC_{B_{n}^{(1)}})$ is a zero-divisor. 
It contradicts the fact that  $K(\CC_{B_{n}^{(1)}})$   is a polynomial ring generated by the classes of fundamental representations (\cite{FR99}).
\QED

Let us denote by $\Irr(\CC^0_{{B_{n}^{(1)}}})$
 the set of isomorphism classes of simple objects in $ \CC^0_{{B_{n}^{(1)}}}$.\
The following is  our main theorem.
\Th
The functor $\tF\col\T_N\to \CC^0_{{B_{n}^{(1)}}}$ induces  a ring  isomorphism
\begin{equation*}
\phi_{\tF} \col K(\T_N)_{q=1} \isoto K(\CC^0_{{B_{n}^{(1)}}})
\end{equation*}
and a bijection between $\Irr(\T_N)_{q=1}$ and
$\Irr(\CC^0_{{B_{n}^{(1)}}})$. 
\enth

 \Proof  We know that
$\Irr(\T_N)_{q=1}$ is a basis of $K(\T_N)_{q=1}$ and
$\Irr(\CC^0_{{B_{n}^{(1)}}})$ is a basis of $K(\CC^0_{{B_{n}^{(1)}}})$.
By Proposition \ref{prop:injective}, the ring homomorphism $\phi_{\tF}$ sends 
$\Irr(\T_N)_{q=1}$ to $\Irr(\CC^0_{{B_{n}^{(1)}}})$ injectively. 
On the other hand, $\phi_{\tF}$ is surjective.
Hence we conclude that $\phi_{\tF}$ is an isomorphism and induces 
a bijection between $\Irr(\T_N)_{q=1}$ and $\Irr(\CC^0_{{B_{n}^{(1)}}})$.
\QED

\section{Relation between  quantum affine algebras of type A and type B}
\subsection{Isomorphisms between Grothendieck rings}
Recall from \cite{KKK13A, KKKO15} that for each $t=1,2$, there  is a functor $\tF^{(t)} : \T_N \To \CC^0_{A^{(t)}_{N-1}} $
which induces  a ring isomorphism
\begin{equation*}
\phi_{\tF^{(t)}} \col K(\T_N)_{q=1} \isoto K(\CC^0_{A^{(t)}_{N-1}}).
\end{equation*}
These ring  isomorphisms $\phi_{\tF^{(t)}}$ $(t=1,2)$ are  also bijective on the sets of classes of simple objects.

 By setting  
  \eq
\phi_t:=\phi_{\tF} \circ \phi_{\tF^{(t)}}^{-1} \qquad (t=1,2), \eneq 
we have
\Th\label{th:main}
There exist  ring isomorphisms
$$\phi_{t}: K(\CC^0_{A^{(t)}_{N-1}}) \isoto K(\CC^0_{{B_{n}^{(1)}}})$$
\enth
which induce bijections between the sets of 
isomorphism classes of simple objects. 

\subsection{Images of modules associated with segments}
In this subsection, we take
\eqn
&&\mathscr S_0(A^{(1)}_{2n-1}) \seteq \{(i,(-q)^{p}) \in \{1,\ldots,2n-1\} \times \cor^{\times} \,;\, p \equiv i+1 \hs{-1ex}\mod \ 2 \}, \label{eq:S0A^(1)_n} \quad \text{and}\\
&&\mathscr S_0(A^{(2)}_{2n-1}) \seteq \{( i,\pm(-q)^{p})
\in \{1,\ldots,n\} \times \cor^{\times} \,;\,  \label{eq:S0A^(2)_2n-1} 
\text{$i  \in \{1,\ldots,n\}$, $p \equiv i+1 \hs{-1ex}\mod  2$} \}. \nonumber  
\label{eq:S0A^(2)_2n} 
\eneqn

The  functors ${\tF^{(t)}}$ $(t=1,2)$ send  the simple module associated with  a segment to a fundamental representation. 
More precisely, we have
\eq \label{eq:untwisted}
\phi_{\tF^{(1)}} ([L(a,b)]) =  [V(\varpi_{b-a+1})_{(-q)^{a+b}}]
\eneq
for $a,b \in \Z$, and 
\eq  \label{eq:twisted}
(\phi_{\tF^{(2)}} \circ \phi_{\tF^{(1)}}^{-1}) ([V(\varpi_i)_{x}]) =
\begin{cases}
  [V(\varpi_i)_{x}]  & \text{for } \ 1 \le i \le n, \\
   [V(\varpi_{N-i})_{-x}] & \text{for } n+1 \le i\le 2n-1. \\
\end{cases}
\eneq
for  every $ (i,x) \in \mathscr S_0(A^{(1)}_{2n-1})$.

Hence the   homomorphisms $\phi_t$ $(t=1,2)$  send  a fundamental representation to the image of the simple module associated with  a segment under $\tF$.

Recall that $N=2n$ and 
\eqn
&&\tF(L(a,a+N-1)) \simeq \cor\quad \text{for all } \ a \in\Z, \\
&&\tF(L(a+N,b+N)) \simeq \tF(L(a,b))_{q^{2N-2}} \quad \text{for all } \ a,b \in \Z. 
\eneqn
Hence the following proposition determines $\tF(L(a,b))$ for 
all $a,b \in\Z$.

\begin{prop} \label{prop:segmentimage}
Let $0\le a \le N-1$ and $0\le b-a\le N-2$.
 Then we have
\bni
\item $\tF(L(0,b)) \simeq V(\varpi_n)_{q^{2b}}$ \quad for  \ $0 \le b \le N-2$.
\item $\tF(L(a,N-1)) \simeq V(\varpi_n)_{q^{2a+N-3}}$ \quad for  \ $1 \le a \le N-1$.
\item
$\tF(L(a,b))$ \\
$\simeq\begin{cases}
V(\varpi_{b-a+1})_{(-1)^{b-a}\qt q^{a+b-2}} & \text{for} \ 1 \le a \le b \le N-2, \ b-a+1 < n, \\
V(\varpi_n)_{q^{2b}} \hconv V(\varpi_n)_{q^{2a+N-3}}
& \text{for $1 \le a \le b \le N-2$, $\ b-a+1 \ge  n$,} \\
V(\varpi_n)_{q^{2a+N-3}} \hconv V(\varpi_n)_{q^{2b-2}}
& \text{for  $1\le a\le N-1<b$, \
$b-a+1 \le  n$,} \\
V(\varpi_{N-b+a-1)})_{(-1)^{b-a}\qt q^{a+b-3}}
& \text{for $1\le a\le N-1<b$,  \
$b-a+1 > n$.}\\
\end{cases}
$
 \end{enumerate}
\begin{proof}
(i) and (ii) follow from Lemma \ref{lem:FK0k}.

\snoi
(iii) Let $1 \le a \le b \le N-2$. Then we have a homomorphism
\eqn
\tF(L(a,b)) \tens \tF(L(b+1,N-1)) \epito \tF(L(a,N-1))\simeq V(\varpi_n)_{q^{2a+N-3}}.
\eneqn
Since 
  $V(\varpi_n)_{q^{2b}}$ is a left dual of $V(\varpi_n)_{q^{2b+N-1}}\simeq \tF(L(b+1,N-1))$,  we have
\eqn
\tF(L(a,b)) \monoto  V(\varpi_n)_{q^{2a+N-3}} \tens V(\varpi_n)_{q^{2b}}. 
\eneqn
It follows that 
\eqn
\tF(L(a,b)) \simeq  V(\varpi_n)_{q^{2b}} \hconv V(\varpi_n)_{q^{2a+N-3}}. 
\eneqn
  If $b-a+1< n$ , then  $2b  \le 2a+N-3$ and  hence by \eqref{eq:Doreynn}, we have
  \eqn 
  V(\varpi_n)_{q^{2b}} \hconv V(\varpi_n)_{q^{2a+N-3}} \simeq V(\varpi_{b-a+1})_{(-1)^{b-a}\qt q^{a+b-2}}.
  \eneqn

By a similar way,  we can obtain   the case when $1 \le a \le N-1 < b$.
\end{proof}
\end{prop}


Let us denote by $V^{(t)}(\varpi_i)_x$  the  fundamental representation corresponding to 
$(i,x) \in \mathscr S_0(A^{(t)}_{2n-1})$ ($t=1,2$).
 Combining \eqref{eq:untwisted}, \eqref{eq:twisted} and Proposition \ref{prop:segmentimage}, we get the following
\begin{corollary}
Let $1\le i \le n$ and let $p\in\Z$ satisfy $p\equiv i+1\mod2$.
\bnum
\item
If $ i-1 \le p \le 2N+i-3$, then we have 
\eqn 
&&\hs{2ex}
\phi_1([V^{(1)}(\varpi_i)_{(-q)^p}]) =
\phi_2([V^{(2)}(\varpi_i)_{(-q)^p}])\\
&&\hs{6ex}=\begin{cases}
[V(\varpi_n)_{q^{2p}}]  & \text{for }  p=i-1, \\
[V(\varpi_i)_{(-1)^{i-1}\qt q^{p-2}}] & \text{for } \ i+1  \le p  \le  2N-i-3, \ i <n, \\
[V(\varpi_n)_{q^{p+n-1}} \hconv V(\varpi_n)_{q^{p+n-2}}] & 
\text{for } \ i+1  \le p  \le  2N-i-3, \ i =n, \\
[V(\varpi_n)_{q^{2p-N-1}}] & \text{for } \ p=2N-i-1, \\
[V(\varpi_n)_{q^{p+N-i-2}} \hconv V(\varpi_n)_{q^{p+i-3}}]  & \text{for } \ 2N-i+1 \le p \le 2N+i-3. \\
\end{cases}
\eneqn
\item If $ N-i-1 \le p \le 3N-i-3$, then we have 
\eqn 
&&\hs{2ex}
\phi_1([V^{(1)}(\varpi_{N-i})_{(-q)^p}]) =
\phi_2 ([V^{(2)}(\varpi_i)_{-(-q)^p}])  \\
&&\hs{6ex}=\begin{cases}
[V(\varpi_n)_{q^{2p}}]  & \text{for }  p=N-i-1, \\
[V(\varpi_n)_{q^{p+N-i-1}}\hconv  V(\varpi_n)_{q^{p+i-2}}] & \text{for } \ N- i+1  \le p  \le  N+i-3, \\
[V(\varpi_n)_{q^{ 2p-N-1}}] & \text{for } \ p=N+i-1, \\
[V(\varpi_i)_{(-1)^{i-1}\qt q^{p-3}}] & \text{for } \ N+i+1 \le p \le 3N-i-3, \ i <n,  \\
[V(\varpi_n)_{q^{p+n-2}} \hconv V(\varpi_n)_{q^{p+n-3}}] &\text{for } \ N+i+1 \le p \le 3N-i-3, \ i =n.
\end{cases}
\eneqn
\ee 
\end{corollary}

Note that we have 
\eqn 
\phi_t ([V^{(t)}(\varpi_i)_{x q^{2N}}]) =  (\phi_t ([V^{(t)}(\varpi_i)_x]))_{q^{2N-2}}
\eneqn
for all $(i,x) \in \mathscr S_0(A^{(t)}_{2n-1})$ ($t=1,2$). Hence the above corollary  determines the maps
$\phi_t$ $(t=1,2)$ on the set of fundamental representations.


\begin{thebibliography}{99}
\nocite{}

\bibitem{AK} T. Akasaka and M. Kashiwara,
{\it Finite-dimensional representations of quantum affine algebras}, 
Publ. RIMS. Kyoto Univ., {\bf33} (1997), 839--867.





\bibitem{CP95} V. Chari and A. Pressley, {\em Quantum affine algebras and their representations}, in Representations of groups (Banff, AB, 1994), CMS Conf. Proc., {\bf16}, Amer. Math. Soc., Providence, RI, 1995, 59--78.



\bibitem{CP96B} \bysame,
{\em Quantum affine algebras and affine Hecke algebras},
Pacific J. Math. {\bf 174} (2) (1996), 295-326.


\bibitem{Che}  I. V. Cherednik,
{\em A new interpretation of Gelfand-Tzetlin bases},
Duke Math. J., {\bf54} (1987), 563-577.




\bibitem{FH11A}
E. Frenkel and D. Hernandez, {\em Langlands duality for finite-dimensional representations
of quantum affine algebras}, Lett. Math. Phys. \textbf{96} (2011) 217--261.

\bibitem{FR99}
E. Frenkel and N. Yu. Reshetikhin,
{\it The q-characters of representations of quantum affine algebras
and deformations of W-algebras, Recent developments in quantum affine algebras and
related topics},
Contemp. Math. {\bf248} (1999), 163--205.


\bibitem{GRV94} V. Ginzburg , N. Reshetikhin and E. Vasserot,
{\em Quantum groups and flag varieties},
A.M.S. Contemp. Math. {\bf 175} (1994), 101-130.


\bibitem{Her101}
D. Hernandez,
{\it Kirillov-Reshetikhin conjecture: the general case},
Int. Math. Res. Not. {\bf2010} (1) (2010), 149--193.



\bibitem{HL10}
D. Hernandez and B. Leclerc, {\em Cluster algebras and quantum affine algebras}, Duke Math. J. \textbf{154} (2010), 265--341




\bibitem{Kac} V. Kac,
\newblock{\it Infinite dimensional Lie algebras},
\newblock{3rd ed., Cambridge University Press,  Cambridge, 1990}.



\bibitem{KKK13A} S.-J. Kang, M. Kashiwara and  M. Kim,
\newblock{\it Symmetric quiver Hecke algebras and $R$-matrices of quantum affine algebras},
\newblock{arXiv:1304.0323 [math.RT]},  to appear in Invent. Math. 

\bibitem{KKK13B} 
\bysame, {\em Symmetric quiver Hecke algebras and R-matrices of
quantum affine algebras II}, Duke Math. J. \textbf{164} (2015), no.~8, 1549--1602.

\bibitem{KKKO14} 
S.-J. Kang, M. Kashiwara,  M. Kim  and   S.-j. Oh,
\newblock{\em Simplicity of heads and socles of tensor products},
Compos. Math. \textbf{151} (2015), no.~2, 377--396.

\bibitem{KKKO15} \bysame,
{\em Symmetric quiver Hecke algebras and R-matrices of quantum affine algebras III}, Proc. Lond. Math. Soc. \textbf{111} (2015), no. 2, 420--444.

\bibitem{KKKO16}
\bysame,
\newblock{\em  Symmetric quiver Hecke algebras and R-matrices of quantum affine algebras IV}, Selecta Math. {\bf 22} (2016), no. 4, 1987--2015.

\bibitem{KKKO15mm}
\bysame,
\newblock{\em Monoidal categorification of cluster algebras II},
\newblock{arXiv:1502.06714}.

\bibitem{KP11} S.-J. Kang and E. Park,
{\em Irreducible modules over Khovanov-Lauda-Rouquier
algebras of type $A_n$ and semistandard tableaux},
J. Algebra {\bf339} (2011), 223--251.




\bibitem{Kas02}
M. Kashiwara,
{\it On level zero representations of quantum affine algebras},
Duke. Math. J. {\bf112} (2002), 117--175.

\bibitem{KO17}
M. Kashiwara, and S.-j. Oh, \emph{Categorical relations between Langlands dual quantum affine algebras: Doubly laced types.}  arXiv:1705.07542  [math.RT].

\bibitem{KP15}
M. Kashiwara and E. Park,
{\em Affinizations and R-matrices for quiver Hecke algebras}, arXiv:1505.03241.




\bibitem{KL09}
M.~Khovanov  and  A. Lauda, \emph{A diagrammatic approach to
categorification of quantum groups
  {I}}, Represent. Theory \textbf{13} (2009), 309--347.



\bibitem{Nak041}
H. Nakajima, {\it Quiver varieties and $t$-analogue of $q$-characters of quantum affine algebras},
Annals of Math. {\bf160} (2004), 1057--1097.



\bibitem{Oh14}
S.-j. Oh, {\it The denominators of normalized R-matrices of types $A^{(2)}_{2n-1}$, $A^{(2)}_{2n}$, $B^{(1)}_n$ and $D^{(2)}_{n+1}$},
Publ. RIMS Kyoto Univ. {\bf51} (2015), 709--744.


\bibitem{R08}
R.~Rouquier, \emph{2-{K}ac-{M}oody algebras},    arXiv:0812.5023v1 [math.RT].


\bibitem{V02}
M. Vazirani, \emph{Parameterizing Hecke algebra modules: Bernstein-Zelevinsky multisegments, Kleshchev multipartitions, and crystal graphs},
Transform. Groups.  {\bf 7 }  (2002),  267--303.

\end{thebibliography}
\end{document}